\documentclass[article]{imsart}

\RequirePackage{amsthm,amsmath,amsfonts,mathabx}
\RequirePackage[numbers]{natbib}
\RequirePackage[colorlinks,citecolor=blue,urlcolor=blue]{hyperref}

\startlocaldefs
\theoremstyle{plain}
\newtheorem{thm}{Theorem}[section]
\newtheorem{lem}{Lemma}[section]
\newtheorem{prop}{Proposition}[section]
\newtheorem{defn}{Definition}[section]
\newtheorem{rem}{Remark}[section]
\endlocaldefs

\begin{document}

\begin{frontmatter}

\title{Consistency of some sequential experimental design strategies for excursion set estimation based on vector-valued Gaussian processes}
\runtitle{Consistency of sequential experimental design strategies}


\author{\fnms{Philip} \snm{Stange}\ead[label=e1]{philip.stange@unibe.ch}}
\and
\author{\fnms{David} \snm{Ginsbourger}\ead[label=e2]{david.ginsbourger@unibe.ch}}

\address{University of Bern\\
Institute of Mathematical Statistics and Actuarial Science\\
Alpeneggstrasse 22, 3012 Bern, Switzerland\\
\printead{e1,e2}}

\runauthor{Stange and Ginsbourger}

\begin{abstract}
We tackle the extension to the vector-valued case of consistency results for Stepwise Uncertainty Reduction sequential experimental design strategies established in \cite{key-1}. 
This lead us in the first place to clarify, assuming a compact index set, how the connection between continuous Gaussian processes and Gaussian measures on the Banach space of continuous functions carries over to vector-valued settings. 
From there, a number of concepts and properties from \cite{key-1} can be readily extended.  
However, vector-valued settings do complicate things for some results, mainly due to the lack of continuity for the pseudo-inverse mapping that affects the conditional mean and covariance function given finitely many pointwise observations. We apply obtained results to the Integrated Bernoulli Variance and the Expected Measure Variance uncertainty functionals employed in \cite{key-3} for the estimation for excursion sets of vector-valued functions.
\end{abstract}





\end{frontmatter}

\newcommand{\bm}[1]{\boldsymbol{#1}}
\global\long\def\E{\mathbb{E}}%
\global\long\def\I{\mathbf{1}}%
\global\long\def\N{\mathbb{N}}%
\global\long\def\R{\mathbb{R}}%
\global\long\def\Rdd{\mathbb{R}^{d\times d}}%
\global\long\def\X{\mathbb{X}}%
\global\long\def\XX{\mathbb{X\times\mathbb{X}}}%
\global\long\def\Cd{}%
\global\long\def\Cdd{\mathbb{}}%
\global\long\def\norm#1{\mathbb{\left\Vert #1\right\Vert }}%
\global\long\def\r#1{\xrightarrow{#1}}%
\global\long\def\wraum{}%
\global\long\def\c#1{\mathcal{#1}}%
\global\long\def\b#1{\mathbb{#1}}%

\global\long\def\t{t}%


\section{Introduction}

Sequential design of experiments is an important statistical area dealing
with the step by step assignment of resources (typically, experiments, measurements, simulations) towards reducing the uncertainty about some quantity of interest. Bect et al. have in \cite{key-1} reinforced the theoretical foundations for the analysis of a large class of strategies that are built according to the stepwise uncertainty reduction (SUR) paradigm. 
This has enabled them to establish some broader consistency results for the considered strategies under the assumption that the function of interest is a sample path of the Gaussian process model used to construct the sequential design. \cite{key-1} is based on the idea that each of the SUR sequential design
strategies involves an uncertainty functional applied to a sequence
of conditional probability distributions such that for any sequential
design the resulting sequence of random variables, that we will denote
by $\left(H_{n}\right)_{n\in\b N}$, is a supermartingale with respect
to the filtration generated by the observations. This is called supermartingale
property of the underlying uncertainty functional. 

In the present work we are establishing extensions of these consistency results to vector-valued (multi-output) settings, with a focus on the situation where multiple quantities are all observed at the same time and may correlate with each other. While such extensions may seem quite natural, so far only very few works have used vector-valued Gaussian processes in theoretical settings, which has motivated us to investigate this aspect and establish links to Gaussian measures on corresponding function spaces. The latter is all the more crucial that the connection between Gaussian processes and Gaussian measures plays a central role in the theoretical constructions used in \cite{key-1} to prove consistency of (scalar-valued) SUR sequential design strategies.

We assume throughout that the function of interest is an element in the
space of continuous functions from a compact metric space $\left(\b X,\mathfrak{d}\right)$
to $\mathbb{R}^{d}$ (for $d\in\b N$), denoted by $\c C\left(\b X;\mathbb{R}^{d}\right)$,
and a sample path of the multivariate Gaussian process $\xi$ that
is used to construct the sequential design. This means $\xi=\left(\xi_{1},...,\xi_{d}\right)$
is a $\mathbb{R}^{d}$-valued Gaussian process (every
$\xi_{i}$ is a $\b R$-valued Gaussian process) with continuous sample
paths defined on a compact metric space $\left(\b X,\mathfrak{d}\right)$. Observations
\[
Z_{n}=\xi\left(X_{n}\right)+\varepsilon_{n}
\]
for $n\geq1$ are to be made sequentially in order to estimate the
quantity of interest. Furthermore, we assume the sequence of
observation errors $\left(\varepsilon_{n}\right)_{n\in\b N}$ to be
independent of the Gaussian process $\xi$ and distributed as independent
centered Gaussian vectors. 

We can then directly take over the definition of a SUR strategy from
\cite{key-1}, which starts with the choice of a ``measure
of residual uncertainty'' for the quantity of interest after $n$
observations
\[
H_{n}=\c H\left(P_{n}^{\xi}\right),
\]
which is a function of the conditional distribution $P_{n}^{\xi}$
of $\xi$ given $\c F_{n}$, where $\mathcal{F}_{n}$ is the $\sigma$-algebra
generated by $X_{1},Z_{1},...,X_{n},Z_{n}$. For $n\geq 0$, the SUR sampling criterion $J_{n}$ associated with $\c H$ is then a function from $\b X$ to $\left[0,\infty\right]$
and defined for $x\in\b X$ as the conditional expectation
\[
J_{n}\left(x\right)=\b E\left[H_{n+1}|\mathcal{F}_{n},X_{n+1}=x\right],
\]
assuming that $H_{n+1}$ is integrable for any choice of $x\in\b X$.
The value of the sampling criterion $J_{n}\left(x\right)$ quantifies
the expected residual uncertainty at time $n+1$ if the next observation
is to be made at $x\in\b X$. The sequential design is then constructed
by minimizing the expected residual uncertainty over $\b X$ 
\[
X_{n+1}\in\underset{x\in\b X}{\text{argmin }}J_{n}\left(x\right).
\]

Note that some of the statements from \cite{key-1} carry over smoothly to the vector-valued case, with proofs needing moderate adjustments and further arguments to hold in the more general setting (e.g. Proposition \ref{prop:3} and \ref{prop:4}). However, other aspects of extending SUR consistency results to vector-valued settings pose novel challenges. One of the key observations in \cite{key-1} is that the sampling criterion is continuous when the covariance
of the underlying Gaussian process is bounded away from zero. Things
become more complicated in the vector-valued case as the pseudo-inverse
mapping presents discontinuities between matrices of different ranks. This affects in turn extending the existence result for SUR sequential design and deserves some more thorough consideration (see Proposition \ref{prop:5}, Lemma \ref{lem:3} and Proposition \ref{prop:7}).

Since only two of the example design strategies from \cite{key-1} have a straightforward extension to vector-valued settings, we focus on them and establish consistency for multi-output extensions of the algorithms of \cite{key-1} dedicated to the excursion probability/ excursion set estimation.
For $f=\left(f_{1},...,f_{d}\right)\in\mathcal{C}\left(\mathbb{X};\mathbb{R}^{d}\right)$
define the  set
\[
\Gamma\left(f\right):=\left\{ u\in\mathbb{X}:f\left(u\right)\in{\bf T}\right\}, 
\]
where ${\bf T}\subset\b R^{d}$ is some closed set.
For the case of orthants ${\bf T}:=\left[\t_{1},\infty\right)\times...\times\left[\t_{d},\infty\right)$,
\[
\Gamma\left(f\right)=\left\{ u\in\mathbb{X}:f_{i}\left(u\right)\geq \t_{i},i\in\left\{ 1,...,d\right\} \right\} .
\]
is called excursion set with excursion threshold
$T=\left(\t_{1},...,\t_{d}\right)^{\top}\in\mathbb{R}^{d}$.
Given a finite measure $\mu$ on $\mathbb{X}$,
the first measure of residual uncertainty in the excursion case is called the integrated Bernoulli variance (IBV) and defined by
\[
H_{n}^{IBV}=\int_{\b X}p_{n}\left(1-p_{n}\right)d\mu,
\]
where $p_{n}\left(x\right)=P\left(\xi\left(x\right)\geq T|\mathcal{F}_{n}\right)$
denotes the excursion probability with respect to $\mathcal{F}_{n}$. 
Still with $\mu$ a finite measure on $\mathbb{X}$, the second measure of residual uncertainty is the variance of the excursion volume (EMV), defined by
\[
H_{n}^{EMV}=\text{Var}\left(\mu\left(\Gamma\left(\xi\right)\right)|\mathcal{F}_{n}\right).
\]
On the first look these criteria may seem to be the same as in \cite{key-1},
since $H_{n}^{IBV}$ and $H_{n}^{EMV}$ are again function to $\mathbb{R}$. In fact, 
the vector-valued aspect is hidden within the definition of the excursion set/ probability and yet complicates theoretical investigations on the aforementioned residual uncertainties
(see Section \ref{sec:5}). 

Both measures of residual uncertainty also appear in \cite{key-3},
where a bivariate Gaussian process $\xi=\left(\xi_{S},\xi_{T}\right)$
is used to jointly model salinity and temperature fields 
to delineate the river plume in a considered domain at the interface between the Fjord of Trondheim and the ocean. Another related setting of learning many tasks simultaneously using kernel methods (multi-task learning) also arises in \cite{key-12} and \cite{key-31} and has turned out
to significantly outperform standard single-task learning methods
in some cases. The multi-output model for Gaussian processes was also
recently studied and encouraged in \cite{key-30}. 

In the next section we will prepare the ground for our theoretical investigations on SUR strategies for the vector-valued case, starting with a review of the separable Banach space $\c C\left(\b X;\mathbb{R}^{d}\right)$ equipped with the supremum
norm $\left\Vert \cdot\right\Vert _{\infty}$. In particular, we will explore the Borel $\sigma$-algebra on this space as well as its dual space $\c C\left(\b X;\mathbb{R}^{d}\right)^{*}$,
see Proposition \ref{prop:2} and Theorem \ref{thm:6} in the Appendix.
This will enable us to define Gaussian random elements in $\c C\left(\b X;\mathbb{R}^{d}\right)$ and to identify continuous $\mathbb{R}^{d}$-valued Gaussian processes as such random elements, see Theorem \ref{thm:1}. The connection will later be crucial for the proofs of the consistency results and the analysis of SUR sequential design, since we will work with the distribution $P^{\xi}$ of the Gaussian process $\xi$ (or more precisely with its conditional distribution given finitely many observations), which will turn out to be a Gaussian measure on $\c C\left(\b X;\mathbb{R}^{d}\right)$.
The connection between continuous $\mathbb{R}^{d}$-valued Gaussian process and Gaussian random elements in $\c C\left(\b X;\mathbb{R}^{d}\right)$ will be tackled with Theorem~\ref{thm:2} and Lemma~\ref{lem:2}, before defining more precisely the statistical model and design problem in Section 3 as introduced in \cite{key-1}. In Section 4 we will discuss uncertainty functionals and some properties that are important for the existence of SUR sequential design (Lemma \ref{lem:3}) and state general sufficient conditions for the consistency of SUR sequential designs, see Proposition \ref{prop:7}. In Section 5 we will finally apply these consistency results to the two SUR sequential designs introduced above. The proofs from Sections \ref{sec:3}, \ref{sec:4} and \ref{sec:5} are postponed to the Appendix.

\section{\label{sec:2}Gaussian processes and Gaussian random elements}

Let $\left(\Omega,\mathcal{F},P\right)$ be the underlying probability
space. In this section we will focus on the connection of multivariate
Gaussian processes and Gaussian random elements in the space of continuous
functions from a compact metric space $\left(\b X,\mathfrak d\right)$ to $\b R^{d}$
that we denote by $\mathcal{C}\left(\mathbb{X};\b R^{d}\right)$ for
some fixed $d\in\mathbb{N}$. For this purpose we also have to investigate
the (continuous) dual space of the latter function space. The most
important statement of this section is Theorem \ref{thm:1}, which
shows that we can identify each continuous multivariate Gaussian process
as Gaussian random element in $\c C\left(\b X;\b R^{d}\right)$ with
respect to its Borel $\sigma$-algebra and vice versa. The theory
of Gaussian random elements and measures is based on \cite{key-20}
and \cite{key-21}. 
\begin{defn}
A Borel probability measure $\nu$ on a (real) Banach space $\left(B,\left\Vert \cdot\right\Vert \right)$ 
is called \emph{Gaussian} if for any bounded linear functional $L\in B^{*}$
the induced measure $\nu\circ L^{-1}$ is a Gaussian measure on $\left(\mathbb{R},\mathcal{B}\left(\mathbb{R}\right)\right)$.
A \textbf{$B$}-valued random element $X:\Omega\rightarrow B$ is
called Gaussian if its distribution is a Gaussian measure on $B$,
i.e. if for any bounded linear functional $L\in B^{*}$ the random
variable $\langle X,L\rangle$ is Gaussian.
\end{defn}

Let us stress that here a Borel probability measure $\nu$ on $\left(\mathbb{R},\mathcal{B}\left(\mathbb{R}\right)\right)$ is called Gaussian if it is either the Dirac measure $\delta_{\mu}$
at a point $\mu \in\b R$ or has a Lebesgue density of the form 
$f_{\nu}:\mathbb{R}\rightarrow\mathbb{R},\;x\mapsto\frac{1}{\sqrt{2\pi\sigma^{2}}}\exp\left(-\frac{\left(x-\mu\right)^{2}}{2\sigma^{2}}\right)$ for some $\mu\in\mathbb{R}$ and $\sigma>0$  (For any Dirac measure we put $\sigma=0$).  

Recall that a random element $X$ with values in a separable Banach space $B$ (with Borel sigma-algebra denoted $\mathcal{B}$) is Bochner-integrable if $X$ is $\mathcal{F}/\mathcal{B}$-measurable  and $\int_{\Omega}\left\Vert X\right\Vert dP<\infty$.
In this case we can define the expectation of the random element $X$
using the Bochner integral
\[
\mathbb{E}\left[X\right]:=\int_{\Omega}XdP=\int_{\Omega}X\left(\omega\right)dP\left(d\omega\right).
\]
$\mathbb{E}\left[X\right]$ is the unique element in $B$ with the
property $\mathbb{E}\left[\left\langle X,L\right\rangle \right]=\left\langle \mathbb{E}\left[X\right],L\right\rangle $
for all $L\in B^{*}$ (where $\langle, \rangle$ stands for the duality Bracket on $B$), i.e. Bochner and Pettis integral coincide. The
proofs and more on the Bochner integral can be found in the first chapters of \cite{key-34}. Note that
we are only interested in Gaussian random elements with values in
separable Banach spaces and these are already Bochner-integrable by
Fernique's Theorem (see \cite{key-2}, Theorem 8.2.1). The following
proposition comes as a special case of Proposition 2.8 in \cite{key-21}.
\begin{prop}
\label{prop:1}Let $\left(B,\left\Vert \cdot\right\Vert \right)$
be a separable Banach space. The characteristic functional of a Gaussian
random element with values $B$ is of the form 
\[
\varphi_{X}\left(L\right):=\mathbb{E}\left[\exp\left(i\left\langle X,L\right\rangle \right)\right]=\exp\left(i\left\langle m_{X},L\right\rangle -\frac{1}{2}\left\langle K_{X}L,L\right\rangle \right),\;L\in B^{*},
\]
where $m_{X}=\mathbb{E}\left[X\right]\in B$ is called the mean of
$X$ and $K_{X}:B^{*}\rightarrow B$ is called the covariance operator
of $X$ and defined by $R_{X}L=\mathbb{E}\left[\left\langle X,L\right\rangle X\right]$.

Conversely, if $X$ is a random element with values in $B$ and characteristic
functional of the above form, where $m_{X}\in B$ and $K_{X}:B^{*}\rightarrow B$
is a positive and symmetric operator, this means $\left\langle KL,L\right\rangle \geq0$
for all $L\in B^{*}$ and $\left\langle KL_{1},L_{2}\right\rangle =\left\langle L_{1},KL_{2}\right\rangle $
for all $L_{1},L_{2}\in B^{*}$, then is $X$ a Gaussian random element
with mean $m_{X}$ and covariance operator $K_{X}$. 

\end{prop}
By the definition of Gaussian random elements it is clear that the
dual space plays an important role if one wants to characterize a
Gaussian random element in some Banach space. We will first describe
the dual space of $\mathcal{C}\left(\mathbb{X};\b R^{d}\right)$ before
we turn to continuous Gaussian processes with values in $\mathbb{R}^{d}$
and finally connect both objects.

It is a well known result that $\mathcal{C}\left(\mathbb{X};\b R^{d}\right)$
is a separable Banach space if we equip it with the supremum norm
\[
\left\Vert f\right\Vert _{\infty}=\sup_{x\in\mathbb{X}}\left|f\left(x\right)\right|_{\max},
\]
where $\left|x\right|_{\max}:=\max_{i\in\left\{ 1,...,d\right\} }\left|x_{i}\right|$
is the maximum norm on $\mathbb{R}^{d}$ (see Theorem 4.19 in \cite{key-10}).
Similar to the space of continuous $\mathbb{R}$-valued functions
it can be shown that for the Borel $\sigma$-algebra it holds
\[
\mathcal{B}\left(\mathcal{C}\left(\mathbb{X};\mathbb{R}^{d}\right)\right)=\sigma\left(\left\{ \delta_{x}:x\in\mathbb{X}\right\} \right),
\]
where $\delta_{x}:\mathcal{C}\left(\mathbb{X};\mathbb{R}^{d}\right)\rightarrow\mathbb{R}^{d}$
by $\delta_{x}\left(f\right)=f\left(x\right)$ are the evaluation
maps for $x\in\mathbb{R}$. Lemma \ref{lem:10} in the Appendix shows
that $\mathcal{C}\left(\mathbb{X};\b R^{d}\right)$ can be interpreted
as the product space 
\[
\mathcal{C}\left(\mathbb{X}\right)\times...\times\mathcal{C}\left(\mathbb{X}\right)=\left\{ \left(f_{i}\right)_{1\leq i\leq d}:f_{i}\in\mathcal{C}\left(\mathbb{X}\right)\text{\text{ for }}1\leq i\leq d\right\} ,
\]
where $\mathcal{C}\left(\mathbb{X}\right)$ is the space of continuous
functions from $\b X$ to $\b R$. We can use this to characterize
the continuous dual space $\mathcal{C}\left(\mathbb{X};\b R^{d}\right)^{*}$.
\begin{prop}
\label{prop:2} Let $\left(\mathbb{X},d\right)$ be a compact metric
space. The dual space of $\mathcal{C}\left(\mathbb{X};\mathbb{R}^{d}\right)$
is isometrically isomorphic to $\mathfrak{M}\left(\b X\right)\times...\times\mathfrak{M}\left(\b X\right)$,
where $\mathfrak{M}\left(\b X\right)$ is the space of finite signed
measures on $\mathbb{X}$ equipped with the Borel $\sigma$-algebra.
This means for every $L\in\mathcal{C}\left(\mathbb{X};\mathbb{R}^{d}\right)^{*}$
there exist finite signed measure $\mu_{i}$ for $i\in\left\{ 1,...,d\right\} $
such that for all $f=\left(f_{1},...,f_{d}\right)\in\mathcal{C}\left(\mathbb{X};\mathbb{R}^{d}\right)$
it holds
\[
L\left(f\right)=\sum_{i=1}^{d}\int_{\b X}f_{i}\left(x\right)\mu_{i}\left(dx\right).
\]
\end{prop}
\begin{proof}
Since $\b X$ is a compact metric space, every measure in $\mathfrak{M}\left(\b X\right)$
is also a Radon measure. The statement that the dual space of $\mathcal{C}\left(\mathbb{X}\right)$
is the space $\mathfrak{M}\left(\b X\right)$ of finite signed measures
is well-known as Riesz--Markov Representation Theorem, see
Chapter 14 in \cite{key-32}. By Lemma \ref{lem:10} 
we know that $\mathcal{C}\left(\mathbb{X}\right)\times...\times\mathcal{C}\left(\mathbb{X}\right)$
and $\mathcal{C}\left(\mathbb{X};\mathbb{R}^{d}\right)$ are isometrically
isomorphic and Theorem \ref{thm:6} shows that $\left(\mathcal{C}\left(\mathbb{X}\right)\times...\times\mathcal{C}\left(\mathbb{X}\right)\right)^{*}$
is isometrically isomorphic to $\mathfrak{M}\left(\b X\right)\times...\times\mathfrak{M}\left(\b X\right)$.
The last statement follows again by Theorem \ref{thm:6}.
\end{proof}
\begin{defn}
Let $\left(\X,d\right)$ be a compact metric space and $\xi=\left(\xi\left(x\right)\right)_{x\in\mathbb{X}}$
a stochastic process with state space $\left(\mathbb{R}^{d},\c B\left(\mathbb{R}^{d}\right)\right)$.
$\xi$ is called a multivariate ($d$-variate, vector-valued) Gaussian
process if the finite-dimensional distributions of $\xi$ are Gaussian,
i.e. if 
\[
\left(\xi\left(x_{1}\right)^{\top},...,\xi\left(x_{n}\right)^{\top}\right)^{\top}
\]
is a Gaussian vector in $\mathbb{R}^{dn}$ for every $n\geq1$
and $x_{1},...,x_{n}\in\mathbb{X}$. 
\end{defn}
Since the multivariate normal distribution is determined by its mean
vector and covariance matrix, the finite-dimensional distributions
of a multivariate Gaussian process $\xi$ are determined by the mean
function 
\[
m:\mathbb{X}\rightarrow\mathbb{R}^{d},\:x\mapsto\mathbb{E}\left[\xi\left(x\right)\right]
\]
and matrix covariance function
\[
k:\mathbb{X}\times\mathbb{X}\rightarrow\mathbb{R}^{d\times d},\:\left(x,y\right)\mapsto\mathbb{E}\left[\left(\xi(x)-m(x)\right)\left(\xi(y)-m(y)\right)^{\top}\right].
\]
Note that we sometimes write $\xi=\left(\xi_{1},...,\xi_{d}\right)$
for a multivariate Gaussian process, where $\xi_{i}\left(x\right):=\pi_{i}\left(\xi\left(x\right)\right)=\xi\left(x\right)_{i}$
is the i-th component of $\xi\left(x\right)$ for $i\in\left\{ 1,...,d\right\} $
and $x\in\b X$. By the definition of a multivariate Gaussian process
this means $\xi_{i}$ is a real-valued Gaussian process for all $i\in\left\{ 1,...,d\right\} $.
Furthermore, we have for the entries of the mean function
\[
m\left(x\right)_{i}=\mathbb{E}\left[\xi_{i}\left(x\right)\right]
\]
for $x\in\b X$ and $i\in\left\{ 1,...,d\right\} $ and for the entries of the matrix covariance function
\[
k(x,y)_{ij}=\mathbb{E}\left[\left(\xi_{i}(x)-m(x)_{i}\right)\left(\xi_{j}(y)-m(y)_{j}\right)\right]
\]
for $x,y\in\mathbb{X}$ and $i,j\in\left\{ 1,...,d\right\} $. Hence
the entries $k(x,y)_{ij}$ of the matrix $k(x,y)$ correspond to the
covariance between the outputs $\xi_{i}(x)$ and $\xi_{j}(y)$ and
describe the degree of correlation or similarity between them.
\begin{lem}
The matrix covariance function $k:\mathbb{X}\times\mathbb{X}\rightarrow\mathbb{R}^{d\times d}$
is
\begin{enumerate}
\item symmetric, that means 
\[
k\left(x,y\right)=k\left(y,x\right)^{\top}
\]
 for all $x,y\in\mathbb{X}$, and
\item positive semi-definite, that means 
\[
\sum_{i,j=1}^{n}\left\langle k\left(x_{i},x_{j}\right)a_{i},a_{j}\right\rangle =\sum_{i,j=1}^{n}a_{j}^{\top}k\left(x_{i},x_{j}\right)a_{i}\geq0
\]
for all $n\geq1$, $a_{i}\in\mathbb{R}^{d}$ not all null and $x_{i}\in\mathbb{X}$ $(1\leq i \leq n)$.
$k$ is called positive definite, if the inequality is strict.
\end{enumerate}
\end{lem}
Consideration of sample path properties makes it possible to think
of Gaussian processes as measurable maps $\xi:\Omega\rightarrow\b S$
from the underlying probability space to a function space $\b S\subset\left(\mathbb{R}^{d}\right)^{\b X}$.
In the following we will see that the induced random element will
also be Gaussian if we consider continuous sample paths. See also
\cite{key-16} for the case $\b X\subset\b R$.
\begin{thm}
\label{thm:1}A multivariate Gaussian process $\xi$ with continuous
sample paths is a Gaussian random element in $\left(\mathcal{C}\left(\mathbb{X};\mathbb{R}^{d}\right),\left\Vert \cdot\right\Vert _{\infty}\right)$
with respect to the Borel $\sigma$-algebra and hence its distribution
is a Gaussian measure on this space . Vice versa, we can find for
every Gaussian measure $\nu$ on $\mathcal{C}\left(\mathbb{X};\mathbb{R}^{d}\right)$
a multivariate Gaussian process with continuous sample paths that
has distribution $\nu$. The distribution of $\xi$ is uniquely determined
by the mean function $m$ and covariance function $k$, so we use
the notation $\xi\sim\mathcal{GP}_{d}\left(m,k\right)$.
\end{thm}

\begin{thm}
\label{thm:2}$\xi$ is a continuous multivariate Gaussian process
with zero mean function and covariance function $k:\mathbb{X}\times\mathbb{X}\rightarrow\mathbb{R}^{d\times d}$
if and only if $\xi$ is a centered Gaussian random element in the
separable Banach space $B$ with covariance operator $K_{\xi}:B^{*}\rightarrow B$.
The covariance operator can be derived from the covariance function
$k$ by 
\[
L\mapsto K_{\xi}L=\left[y\in\mathbb{X}\mapsto\left(\sum_{k=1}^{d}\int_{\mathbb{X}}k\left(x,y\right)_{kl}\mu_{k}\left(dx\right)\right)_{l=1,...,d}\right],
\]
where $\mu_{k}$ are the finite signed measures from Proposition \ref{prop:2}).
Given the covariance operator $K_{\xi}:B^{*}\rightarrow B$ we can
derive the covariance function by 
\[
\left(x,y\right)\mapsto\left(\left\langle K_{\xi}\delta_{x}^{k},\delta_{y}^{l}\right\rangle \right)_{k,l=1}^{d},
\]
where $\delta_{x}^{k}:=\pi_{k}\circ\delta_{x}$ and $\delta_{y}^{l}:=\pi_{l}\circ\delta_{y}$
for $x,y\in\mathbb{X}$ and $k,l\in\left\{ 1,...,d\right\} $ .
\end{thm}
\begin{proof}
The statement regarding Gaussian processes and Gaussian random elements in the Banach space $\mathcal{C}\left(\mathbb{X};\mathbb{R}^{d}\right)$
follows from Theorem \ref{thm:1}, so we focus on the connection
between the covariance function and the covariance operator. On the
one hand we have for $L,L'\in B^{*}$ by Proposition \ref{prop:2}
the existence of finite signed measures $\mu_{k}$ and $\mu'_{l}$
for $k,l\in\left\{ 1,...,d\right\} $ on $\mathbb{X}$ equipped with
the Borel $\sigma$-algebra such that

\begin{align*}
\text{\ensuremath{\left\langle K_{\xi}L,L'\right\rangle }} & =\mathbb{E}\left[\left\langle \xi,L\right\rangle \left\langle \xi,L'\right\rangle \right]\\
 & =\mathbb{E}\left[\left(\sum_{k=1}^{d}\int_{\mathbb{X}}\xi_{k}\left(x\right)\mu_{k}\left(dx\right)\right)\left(\sum_{l=1}^{d}\int_{\mathbb{X}}\xi_{l}\left(y\right)\mu'_{l}\left(dy\right)\right)\right]\\
 & =\mathbb{E}\left[\sum_{k,l=1}^{d}\int_{\mathbb{X}}\left(\int_{\mathbb{X}}\xi_{k}\left(x\right)\xi_{l}\left(y\right)\mu_{k}\left(dx\right)\right)\mu'_{l}\left(dy\right)\right]\\
 & =\sum_{k,l=1}^{d}\int_{\mathbb{X}}\left(\int_{\mathbb{X}}\mathbb{E}\left[\xi_{k}\left(x\right)\xi_{l}\left(y\right)\right]\mu_{k}\left(dx\right)\right)\mu'_{l}\left(dy\right)\\
 & =\sum_{k,l=1}^{d}\int_{\mathbb{X}}\left(\int_{\mathbb{X}}k\left(x,y\right)_{kl}\mu_{k}\left(dx\right)\right)\mu'_{l}\left(dy\right).
\end{align*}
Hence the covariance operator can be obtained from the covariance
function by integrating.

Let $\delta_{x},\delta_{y}$ with $x,y\in\b X$ be the continuous
evaluation maps defined by $\delta_{x}:B\rightarrow\b R^{d}$, $\delta_{x}\left(f\right)=f\left(x\right)$.
If we define $\delta_{x}^{k}:=\pi_{k}\circ\delta_{x}$ and $\delta_{y}^{l}:=\pi_{l}\circ\delta_{y}$
for $x,y\in\mathbb{X}$ and $k,l\in\left\{ 1,...,d\right\} $, this
yields that $\delta_{x}^{k}$ and $\delta_{y}^{l}$ are continuous
linear functions from $B$ to $\b R$, since the projection maps $\pi_{k}:X_{1}\times...\times X_{d}\rightarrow X_{k}$,
$\pi_{k}\left(x\right)=x_{k}$ on product spaces are linear and continuous.
Hence we conclude $\delta_{x}^{k},\delta_{y}^{l}\in B^{*}$ with 
\begin{align*}
\left\langle K_{\xi}\delta_{x}^{k},\delta_{y}^{l}\right\rangle  & =\mathbb{E}\left[\left\langle \xi,\delta_{x}^{k}\right\rangle \left\langle \xi,\delta_{y}^{l}\right\rangle \right]\\
 & =\mathbb{E}\left[\xi_{k}\left(x\right)\xi_{l}\left(y\right)\right]\\
 & =k\left(x,y\right)_{kl}.
\end{align*}
\end{proof}
\begin{rem}
The existence of versions of a Gaussian process $\xi$ with sample
path properties such as continuity are connected to the covariance
function $k(s,t)=\text{Cov}(\xi\left(s\right),\xi\left(t\right))$
of the process. This means not for every mean function $m$ and covariance
function $k$ we have the existence of a continuous multivariate Gaussian
process, not even when $m$ and $k$ are continuous. However, a criterion
for the existence of a continuous version is given by the Kolmogorov-Chentsov
Theorem: Assume $\mathbb{X}\subseteq\mathbb{R}^{d}$ and suppose for
the covariance function $k$ it holds 
\[
\E\left[\left\Vert \xi\left(x\right)-\xi\left(y\right)\right\Vert _{2}^{2}\right]=\sum_{i=1}^{d}\E\left[\left|\xi_{i}\left(x\right)-\xi_{i}\left(y\right)\right|^{2}\right]\leq\frac{C}{\left|\text{log}\left\Vert x-y\right\Vert \right|^{1+\gamma}},
\]
where $C>0,\gamma>0$ are some constants. Then there exists a Gaussian
process $\xi=\left(\xi_{1},...,\xi_{d}\right)$ with mean $m\equiv0$
and covariance function $k$ that is a stochastic process with continuous
sample paths. For a proof see Theorem 3.4.1 in \cite{key-13}. 
\end{rem}
\begin{lem}
\label{lem:2}Let $\xi$ be a multivariate Gaussian process with continuous
sample paths. Then the mean function $m$ and covariance function
$k$ are continuous. 
\end{lem}

\section{Conditioning on finitely many observations\label{sec:3}}

Now that we have the grounding, we can revisit the construction from \cite{key-1} around conditioning on finitely many
observations. Many properties carry over, but we still require
some careful thoughts in places. Especially Propositions \ref{prop:4}
and \ref{prop:5} need some additional arguments if the underlying
Gaussian process has values in $\mathbb{R}^{d}$. Proposition \ref{prop:5}
comes as a surprise and will be the reason that the sample criterion
has a more complicated discontinuity structure in the vector-valued case
than in the scalar-valued case. 
All profs has been moved to Appendix \ref{subsec:app:sec3}.

We will assume that
\begin{enumerate}
\item $\left(\mathbb{X},\mathfrak{d}\right)$ is a compact metric space,
\item $\xi=\left(\xi\left(x\right)\right)_{x\in\mathbb{X}}$ is a $d$-variate
Gaussian process on the probability space $\left(\Omega,\mathcal{F},P\right)$,
with mean function $m$ and covariance function $k$,
\item $\xi$ has continuous sample paths,
\end{enumerate}
and concentrate on the following model.

$\xi$ can be observed at sequentially selected design points $X_{1},X_{2,}...$
with additive independent heteroscedastic Gaussian noise. This means
pointwise observations $Z_{k}$ for $k=1,2,...$ are given by

\[
Z_{k}=\xi\left(X_{k}\right)+\tau\left(X_{k}\right)U_{k},
\]
where $\left(U_{k}\right)_{k\in\b N}$ denotes a sequence of independent
and $\c N_{d}\left(0,I_{d}\right)$-distributed random vectors, that
are also independent of $\xi$, and $\tau:\b X\rightarrow\Rdd$ denotes
a known continuous function. Then $\c T:\b X\rightarrow\Rdd$, $\c T\left(x\right):=\tau\left(x\right)\tau\left(x\right)^{\top}$
is also continuous and $\c T\left(x\right)$ is symmetric and positive
semi-definite for every $x\in\b X$.

Furthermore, we define the filtration $\left(\mathcal{F}_{n}\right)_{n\geq0}$
by 
\[
\mathcal{F}_{n}:=\sigma\left(\left\{ \bigcup_{i=1}^{n}\left(X_{i},Z_{i}\right)\right\} \right)
\]
for $n\geq1$ and set $\c F_{0}$ to be the trivial $\sigma$-algebra.
$\mathcal{F}_{n}$ is the $\sigma$-algebra generated by the first
$n$ sequential design points and $n$ according pointwise observations
$X_{1},Z_{1}$, $X_{2},Z_{2}$, ..., $X_{n},Z_{n}$ and we have $\mathcal{F}_{n}\subseteq\mathcal{F}_{m}$
for $n\leq m$. We finally define 
\[
\mathcal{F}_{\infty}:=\sigma\left(\bigcup_{n\geq1}\mathcal{F}_{n}\right)\subset\mathcal{F}.
\]

\begin{defn}
A sequence $\left(X_{n}\right)_{n\geq1}$ is called \emph{sequential
design} if $X_{n}$ is $\mathcal{F}_{n-1}$-measurable for all $n\geq1$.
\end{defn}
\begin{defn}
For $A\in\mathbb{R}^{n\times m}$ the (Moore-Penrose) pseudo-inverse
of $A$ is defined as the matrix $A^{\dagger}\in\mathbb{R}^{m\times n}$
satisfying the properties
\begin{enumerate}
\item $AA^{\dagger}A=A$ and $A^{\dagger}AA^{\dagger}=A^{\dagger}$, 
\item $\left(AA^{\dagger}\right)^{\top}=AA^{\dagger}$ and $\left(A^{\dagger}A\right)^{\top}=A^{\dagger}A$.
\end{enumerate}
\end{defn}
\begin{rem}
\begin{enumerate}
\item By the Theorem of Moore and Penrose the pseudo-inverse always exists
and is unique (see \cite{key-33}).
\item If $A$ is a square matrix with full rank then $A^{\dagger}=A^{-1}$
(see \cite{key-33}).
\item The mapping $A\mapsto A^{\dagger}$ is measurable (see \cite{key-6}).
\item In contrast to the usual matrix inversion mapping $A\mapsto A^{-1}$
for invertible matrices, the pseudo-inverse mapping $A\mapsto A^{\dagger}$
is in general not continuous. However, continuity of this mapping
is provided on sets with constant matrix rank. This means $A_{n}^{\dagger}\rightarrow A^{\dagger}$,
if $A_{n}\rightarrow A$ and there exists $n_{0}\in\b N$ such that
$\text{rank}\left(A_{n}\right)=\text{rank}\left(A\right)$ for all
$n\geq n_{0}$ (see \cite{key-8,key-7}).
\end{enumerate}
\end{rem}
\begin{thm}
\label{thm:3}For any $\xi\sim\mathcal{GP}_{d}\left(m,k\right)$ with
$\left(m,k\right)\in\Theta$, $\bm{X}_{n}=\left(X_{1},...,X_{n}\right)\in\mathbb{X}^{n}$,
$\bm{Z}_{n}=\left(Z_{1,}...,Z_{n}\right)\in\mathbb{R}^{d\times n}$,
as defined above, the conditional mean and covariance function of
$\xi$ given $\bm{Z}_{n}=\bm{z}_{n}$ and assuming a deterministic
design $\bm{X}_{n}=\bm{x}_{n}$ are given by 
\[
m_{n}(x;\bm{x}_{n},\bm{z}_{n})=m\left(x\right)+K\left(x,\bm{x}_{n}\right)\Sigma\left(\bm{x}_{n}\right){}^{\dagger}\left(\text{vec}\left(\bm{z}_{n}\right)-\text{vec}\left(m\left(\bm{x}_{n}\right)\right)\right),
\]
\[
k_{n}\left(x,y;\bm{x}_{n}\right)=k\left(x,y\right)-K\left(x,\bm{x}_{n}\right)\Sigma\left(\bm{x}_{n}\right){}^{\dagger}K\left(y,\bm{x}_{n}\right){}^{\top},
\]
where we define $\ensuremath{\Sigma}\left(x\right):=k(x,x)+\c T\left(x\right)$
for $x\in\b X$ and use the matrix convention
\begin{align*}
m\left(\bm{x}_{n}\right) & :=\left(\begin{array}{ccc}
m\left(x_{1}\right) & \cdots & m\left(x_{n}\right)\end{array}\right)\in\mathbb{R}^{d\times n},\\
K\left(x,\bm{x}_{n}\right) & :=\left(\begin{array}{ccc}
k\left(x,x_{1}\right) & \cdots & k\left(x,x_{n}\right)\end{array}\right)\in\mathbb{R}^{d\times nd},\\
\Sigma\left(\bm{x}_{n}\right) & :=K\left(\bm{x}_{n}\right)+\c T\left(\bm{x}_{n}\right)\in\mathbb{R}^{nd\times nd},
\end{align*}
with 
\[
K\left(\bm{x}_{n}\right):=\left(\begin{array}{cccc}
k(x_{1},x_{1}) & k(x_{1},x_{2}) & \cdots & k(x_{1},x_{n})\\
k(x_{2},x_{1}) & \ddots &  & \vdots\\
\vdots &  & \ddots & k(x_{n-1},x_{n})\\
k(x_{n},x_{1}) & \cdots & k(x_{n},x_{n-1}) & k(x_{n},x_{n})
\end{array}\right),
\]
\[
\c T\left(\bm{x}_{n}\right):=\left(\begin{array}{cccc}
\c T\left(x_{1}\right) & 0 & \cdots & 0\\
0 & \ddots &  & \vdots\\
\vdots &  & \ddots & 0\\
0 & \cdots & 0 & \c T\left(x_{n}\right)
\end{array}\right).
\]
\end{thm}
\begin{rem}
\begin{enumerate}
\item For the case $d=1$ the  formula reduces to the form
\[
m_{n}\left(x;\bm{x}_{n},\bm{z}_{n}\right)=m\left(x\right)+K\left(x,\bm{x}_{n}\right)\Sigma\left(\bm{x}_{n}\right){}^{\dagger}\left(\bm{z}_{n}-m\left(\bm{x}_{n}\right)\right)^{\top},
\]
\[
k_{n}\left(x,y;\bm{x}_{n}\right)=k\left(x,y\right)-K\left(x,\bm{x}_{n}\right)\Sigma\left(\bm{x}_{n}\right){}^{\dagger}K\left(y,\bm{x}_{n}\right){}^{\top}
\]
with $\Sigma\left(\bm{x}_{n}\right)=\left(k\left(x_{i},x_{j}\right)+\mathcal{T}(x_{i})\delta_{i,j}\right)_{1\leq i,j\leq n}\in\mathbb{R}^{n\times n}$.
Hence the expressions for $m_{n}$ and $k_{n}$ are consistent with
the ones for $\b R$-valued Gaussian processes as in \cite{key-1}.
\item Conditionally to $\c F_{n}$, the next observation $Z_{n+1}$ follows
a multivariate normal distribution: $Z_{n+1}|\mathcal{F}_{n}\sim\c N_{d}\left(m_{n}\left(X_{n+1}\right),\Sigma_{n}\left(X_{n+1}\right)\right)$,
where 
\[\Sigma_{n}\left(x\right):=k_{n}\left(x,x\right)+\c T\left(x\right).\]
\item For the case $n=0$ we have $\Sigma\left(x\right)=\Sigma_{0}\left(x\right)=k\left(x,x\right)+\mathcal{T}\left(x\right)\in\mathbb{R}^{d\times d}$
as expected for our model defined above. 
\item The conditional mean $m_{n}$ (or posterior mean in Bayesian statistics)
is also called Kriging predictor or rather co-Kriging in our case,
since we have the joint Kriging of multiple data inputs. For the univariate
response setting, there exists abundant literature to it (see \cite{key-40,key-38,key-41}
and the references therein). For the multivariate response see also
\cite{key-12,key-30}.
\end{enumerate}
\end{rem}
As shown in Section \ref{sec:2} we can think of the multivariate
Gaussian process $\xi$ as a Gaussian random element in the separable
Banach Space 
\[
\mathbb{S}:=\mathcal{C}\left(\mathbb{X};\mathbb{R}^{d}\right)
\]
equipped with the supremum (or uniform) norm 
\[
\left\Vert f\right\Vert _{\infty}:=\sup_{x\in\mathbb{X}}\left\Vert f(x)\right\Vert _{\max}
\]
and Borel $\sigma$-algebra 
\[
\c S:=\mathcal{B}\left(\mathbb{S}\right).
\]
Let furthermore $\mathbb{M}$ be the \emph{space of Gaussian measures
on $\mathbb{S}$}. We equip $\mathbb{M}$ with the $\sigma$-algebra
$\mathcal{M}$ generated by the evaluation maps $\pi_{A}:\nu\rightarrow\nu\left(A\right)$
for $A\in\mathcal{S}$. As proven in Theorem \ref{thm:1} we have
the following connections:

Any measure $\nu\in\mathbb{M}$ corresponds to the distribution $P^{\xi}$
of some continuous multivariate Gaussian Process $\xi$ with mean
$m$ and covariance function $k$. Hence we can write 
\[
\nu=P^{\xi}=\mathcal{GP}_{d}\left(m,k\right),
\]
 since the distribution is uniquely determined by $m$ and $k$.

On the other hand the probability distribution $P^{\xi}=\mathcal{GP}_{d}\left(m,k\right)$
of some continuous multivariate Gaussian process is a Gaussian measure
on $\mathbb{S}$ and hence an element in $\mathbb{M}$. 
\begin{defn}
Given a Gaussian random element $\xi$ in $\mathbb{S}$, we will denote
by $\mathfrak{P}\left(\xi\right)$ the \emph{set of all Gaussian conditional
distributions of $\xi$. }That is the set of Gaussian random measures
$\bm{\nu}$ such that $\bm{\nu}=P\left(\xi\in\cdot|\mathcal{F}'\right)$
for some $\sigma$-algebra $\mathcal{F}'\subset\mathcal{F}$.
\end{defn}
\begin{rem}
\begin{enumerate}
\item Note that we use a bold letter $\bm{\nu}$ to denote a random element
in $\mathbb{M}$ (Gaussian random measure) and a normal letter $\nu$
to denote a point in the space $\mathbb{M}$ (Gaussian measure). 
\item $\bm{\nu}=P\left(\xi\in\cdot|\mathcal{F}'\right)$ is not necessarily
Gaussian for an arbitrary $\sigma$-algebra $\mathcal{F}'\subset\mathcal{F}$
and hence not always a random element in $\mathbb{M}$ . However,
the next Proposition shows that it holds for the $\sigma$-algebra
$\mathcal{F}_{n}$ generated by a sequential design with corresponding pointwise observations.
\end{enumerate}
\end{rem}
\begin{prop}
\label{prop:3}For all $n\geq1$, there exists a measurable mapping
\begin{align*}
\mathbb{X}^{n}\times\mathbb{R}^{d\times n}\times\mathbb{M} & \rightarrow\mathbb{M}\\
\left({\bf x}_{n},{\bf z}_{n},\nu\right) & \mapsto\text{Cond}{}_{x_{1},z_{1},...,x_{n},z_{n}}\left(\nu\right),
\end{align*}
where we write ${\bf x}_{n}=\left(x_{1},...,x_{n}\right)$ and ${\bf z}_{n}=\left(z_{1},...,z_{n}\right)$,
such that for any $\nu=P^{\xi}\in\mathbb{M}$ and any sequential design
$\left(X_{n}\right)_{n\geq1}$ with pointwise observations $\left(Z_{n}\right)_{n\geq1}$
the Gaussian random measure $\text{Cond}{}_{X_{1},Z_{1},...,X_{n},Z_{n}}\left(P^{\xi}\right)$
is a conditional distribution of $\xi$ given the $\sigma$-algebra
$\mathcal{F}_{n}$ generated by the pointwise observations $X_{1},Z_{1},...,X_{n},Z_{n}$.
We will denote $\text{Cond}{}_{X_{1},Z_{1},...,X_{n},Z_{n}}\left(P^{\xi}\right)$
by $P_{n}^{\xi}$.
\end{prop}
\begin{rem}
For any $\xi\sim\mathcal{GP}_{d}\left(m,k\right)$ and $\sigma$-algebra
$\mathcal{F}_{n}$, generated by finite observations $X_{1},Z_{1},...,X_{n},Z_{n}$,
we have
\[
P_{n}^{\xi}:=\mathcal{GP}_{d}\left(m_{n},k_{n}\right)=\text{Cond}{}_{X_{1},Z_{1},...,X_{n},Z_{n}}\left(P^{\xi}\right)=P\left(\xi\in\cdot|\mathcal{F}_{n}\right),
\]
which can be seen as a $\mathcal{F}_{n}$-measurable random element
in $\left(\mathbb{M},\mathcal{M}\right)$ by Proposition \ref{prop:3}
and hence $P_{n}^{\xi}\in\mathfrak{P}\left(\xi\right)$ for all $n\in\b N$.
The $\c F_{n}$-measurable (and random) conditional mean function
$m_{n}$ and conditional covariance function $k_{n}$ are given as
in Theorem \ref{thm:3}.
\end{rem}
\begin{defn}
Let $\left(\nu_{n}=\mathcal{GP}_{d}\left(m_{n},k_{n}\right)\right)_{n\geq1}$
be a sequence of Gaussian measures in $\mathbb{M}$. We will say that
$\left(\nu_{n}\right)_{n\geq1}$ converges to $\nu_{\infty}\in\mathbb{M}$
if
\[
m_{n}\rightarrow m_{\infty}\quad\text{in}\quad\mathcal{C}\left(\mathbb{X};\mathbb{R}^{d}\right),
\]
\[
k_{n}\rightarrow k_{\infty}\quad\text{in}\quad\mathcal{C}\left(\mathbb{X}\times\mathbb{X};\mathbb{R}^{d\times d}\right)
\]
with respect to the corresponding supremum norms $\left\Vert \cdot\right\Vert _{\infty}$
on the function spaces. Notation: $\nu_{n}\rightarrow\nu_{\infty}$.
\end{defn}
\begin{rem}
The above notion of convergence is weaker than weak convergence of
measures. Indeed, for weak convergence we would also need that the
sequence $\left(\nu_{n}\right)_{n\geq1}$ is tight (see \cite{key-21}),
which is not the case in general. 
\end{rem}
\begin{prop}
\label{prop:4}Let $\mathcal{F}_{\infty}$ be the $\sigma$-algebra
generated by $\bigcup_{n\geq1}\mathcal{F}_{n}$. For any Gaussian
random element $\xi$ in $\mathbb{S}$, defined on any probability
space $\left(\Omega,\mathcal{F},\mathcal{P}\right)$, and for any
sequential design $\left(X_{n}\right)_{n\geq1}$, the conditional
distribution of $\xi$ given $\mathcal{F}_{\infty}$ admits a version
$P_{\infty}^{\xi}$ which is an $\mathcal{F}_{\infty}$-measurable
random element in $\mathbb{M}$, and it holds
\[
P_{n}^{\xi}\xrightarrow[n\rightarrow\infty]{a.s.}P_{\infty}^{\xi}.
\]
\end{prop}
\begin{prop}
\label{prop:5}Let $\nu=\mathcal{GP}_{d}\left(m_{\nu},k_{\nu}\right)\in\mathbb{M}$,
$\Sigma_{\nu}\left(x\right):=k_{\nu}\left(x,x\right)+\c T\left(x\right)\in\Rdd$
and assume $\left(x_{i},z_{i}\right)\xrightarrow{}\left(x,z\right)$
as $i\rightarrow\infty$ in $\mathbb{X}\times\mathbb{R}^{d}$ with $x_{i},x$ in the set 
\[
C_{k}:=\left\{ x\in\mathbb{X}:\text{rank}\left(\Sigma_{\nu}\left(x\right)\right)=k\right\}
\]
and $z_{i},z\in\mathbb{R}^{d}$ for $i\in\mathbb{N}$ and $k\in\left\{ 1,...,d\right\} $.
Then we have
\[
\text{Cond}{}_{x_{i},z_{i}}\left(\nu\right)\xrightarrow[i\rightarrow\infty]{}\text{Cond}{}_{x,z}\left(\nu\right).
\]
\end{prop}

\begin{rem}
    The above Proposition illustrates an important difficulty that arises when one turns from a regular Gaussian process in $\mathcal{GP}\left(m,k\right)$ to a multivariate Gaussian process in $\mathcal{GP}_{d}\left(m,k\right)$. 
    For a Gaussian process in $\mathcal{GP}\left(m,k\right)$ and just one observation $({\bf x}_{1}, {\bf z}_{1}) = (x,z)$, the convergence mentioned above only depends on the inverse of a scalar, whereas for multivariate Gaussian processes in $\mathcal{GP}_{d}\left(m,k\right)$ we have already to deal with matrix inversion. This yields some problems for the limits as $x_i \rightarrow x$.
    
    However, this is an important difference to the limits of the co-Kriging $m_n$ and conditional variance $k_n$ that are taken on the observations $({\bf x}_{n}, {\bf z}_{n})$ as $n \rightarrow \infty$, where we can use more abstract interpretations of both functions and do not have to rely on the explicit form as in Theorem \ref{thm:3}.
\end{rem}

\section{SUR sequential design and its existence in the multivariate setting\label{sec:4}}

We will start by recalling some definitions from \cite{key-1} that
can instantly be extended to our multivariate setting, since they
only depend on the space $\mathbb{M}$ of Gaussian measures on $\mathbb{S}=\mathcal{C}\left(\mathbb{X};\mathbb{R}^{d}\right)$.
However, the existence of SUR sequential design in the multivariate
case comes with some pitfalls caused by Proposition \ref{prop:5}
in the previous section. We will nevertheless prove that under some
special assumptions the SUR sequential design indeed exists. All proofs
have been moved to Appendix \ref{subsec:app:sec4}.
\begin{defn}
An uncertainty functional on $\mathbb{M}$ is a measurable
function 
\[
\mathcal{H}:\mathbb{M}\rightarrow\left[0,\infty\right)
\]
with $\min_{\nu\in\mathbb{M}}\mathcal{H}\left(\nu\right)=0$. The residual uncertainty after $n$ observations, for a
Gaussian random element $\xi$ in $\mathbb{S}$ and a sequential design
$\left(X_{n}\right)_{n\geq1}$, is defined as the $\mathcal{F}_{n}$-measurable
random variable 
\[
H_{n}:=\mathcal{H}\left(P_{n}^{\xi}\right)
\]
for $n\geq0$.
\end{defn}
\begin{defn}
Let $\mathcal{H}$ be an uncertainty functional on $\mathbb{M}$.
\begin{enumerate}
\item $\mathcal{H}$ has the supermartingale property if for any Gaussian
random element $\xi$ in $\mathbb{S}$, defined on any probability
space $\left(\Omega,\mathcal{F},P\right)$, and any sequential design
$\left(X_{n}\right)_{n\geq1}$ the sequence $\left(H_{n}\right)_{n\geq0}$
defined by 
\[
H_{n}:=\mathcal{H}\left(P_{n}^{\xi}\right)
\]
is a $\left(\mathcal{F}_{n}\right)_{n\geq0}$-supermartingale.
\item $\mathcal{H}$ is $\mathfrak{P}$-uniformly integrable if for
any Gaussian random element $\xi$ in $\mathbb{S}$, defined on any
probability space, the family $\left(\mathcal{H}\left(\bm{\nu}\right)\right)_{\bm{\nu}\in\mathfrak{P}\left(\xi\right)}$
is uniformly integrable. 
\item $\mathcal{H}$ is $\mathfrak{P}$-continuous if for any Gaussian random
element $\xi$ in $\mathbb{S}$, defined on any probability space,
and any sequence of random measures $\left(\bm{\nu}_{n}\right)_{n\geq1}\subset\mathfrak{P}\left(\xi\right)$
such that $\bm{\nu}_{n}\xrightarrow[n\rightarrow\infty]{a.s.}\bm{\nu}_{\infty}\in\mathfrak{P}\left(\xi\right)$
it holds 
\[
\mathcal{H}\left(\bm{\nu}_{n}\right)\xrightarrow[n\rightarrow\infty]{a.s.}\mathcal{H}\left(\bm{\nu}_{\infty}\right).
\]
\end{enumerate}
\end{defn}
For the definition of stepwise uncertainty reduction (SUR) sequential
design strategies we need to define some important functionals on
$\mathbb{M}$. For any $x\in\mathbb{X}$ observe the mapping $\mathcal{J}_{x}:\mathbb{M}\rightarrow\left[0,\infty\right]$
defined by
\begin{align*}
\mathcal{J}_{x}\left(\nu\right) & =\int_{\mathbb{R}^{d}}\int_{\mathbb{S}}\mathcal{H}\left(\text{Cond}{}_{x,f\left(x\right)+\tau\left(x\right)u}\left(\nu\right)\right)\nu\left(df\right)\phi_{d}\left(u\right)du\\
 & =\int_{\mathbb{R}^{d}}\mathcal{H}\left(\text{Cond}{}_{x,m_{\nu}\left(x\right)+\Sigma_{\nu}\left(x\right)^{\frac{1}{2}}u}\left(\nu\right)\right)\phi_{d}\left(u\right)du,
\end{align*}
where $\Sigma_{\nu}\left(x\right)^{\frac{1}{2}}\in\mathbb{R}^{d\times d}$
is the unique symmetric and positive semi-definite square root matrix
of $\Sigma_{\nu}\left(x\right)=k_{\nu}\left(x,x\right)+\mathcal{T}\left(x\right)$,
$\phi_{d}$ is the probability density function of the $\mathcal{N}_{d}\left(0,I_{d}\right)$-distribution
and $\nu=\c{GP}_{d}\left(m_{\nu},k_{\nu}\right)$. 
\begin{prop}
\label{prop:6}The mapping 
\[
\mathcal{J}:\mathbb{X}\times\mathbb{M}\rightarrow\left[0,\infty\right],\,\left(x,\nu\right)\mapsto\mathcal{J}_{x}\left(\nu\right)
\]
is $\mathcal{B}\left(\mathbb{X}\right)\otimes\mathcal{M}$-measurable.
\end{prop}
\begin{defn}
Let $\xi$ be a Gaussian random element in $\mathbb{S}$, $\left(X_{n}\right)_{n\geq1}$
be a sequential design and $\mathcal{F}_{n}$ the $\sigma$-algebra
generated by $X_{1},Z_{1},...,X_{n},Z_{n}$. The SUR sampling criterion
$J_{n}$ associated to an uncertainty functional $\mathcal{H}$ on
$\mathbb{M}$ is defined as the function $J_{n}:\b X\rightarrow\left[0,\infty\right]$,
where 
\[
J_{n}\left(x\right):=\mathcal{J}_{x}\left(P_{n}^{\xi}\right)=\mathbb{E}_{n}\left[\mathcal{H}\left(\text{Cond}{}_{x,Z(x)}\left(P_{n}^{\xi}\right)\right)\right]:=\mathbb{E}\left[\mathcal{H}\left(\text{Cond}{}_{x,Z(x)}\left(P_{n}^{\xi}\right)\right)|\mathcal{F}_{n}\right]
\]
with $Z(x)=\xi\left(x\right)+\tau\left(x\right)U$ and $U\sim\c N_{d}\left(0,I_{d}\right)$
independent of $\xi,U_{1},...,U_{n}$ as defined in the introduction
of Section \ref{sec:3}.
\begin{enumerate}
\item $\left(X_{n}\right)_{n\geq1}$ is called a SUR sequential design associated
to the uncertainty functional $\mathcal{H}$, if 
\[
X_{n+1}\in\underset{x\in\b X}{\text{argmin }}J_{n}\left(x\right)
\]
for all $n\geq n_{0}$ with $n_{0}\in\mathbb{N}$. 
\item Let $\left(\varepsilon_{n}\right)_{n\in\mathbb{N}}$ be a sequence
of non-negative real numbers with $\epsilon_{n}\rightarrow0$ as $n\rightarrow\infty$.
$\left(X_{n}\right)_{n\geq1}$ is called an $\varepsilon$-quasi SUR
sequential design associated to the uncertainty functional $\mathcal{H}$,
if it holds
\[
J_{n}\left(X_{n+1}\right)\leq\inf_{x\in\mathbb{X}}J_{n}\left(x\right)+\varepsilon_{n}
\]
for all $n\geq n_{0}$ with $n_{0}\in\mathbb{N}$. 
\end{enumerate}
\end{defn}
\begin{lem}
\label{lem:3}Let $\mathcal{H}$ be a measurable uncertainty functional
on $\mathbb{M}$ that is $\mathfrak{P}$-continuous, $\mathfrak{P}$-uniformly
integrable and has the supermartingale property.
\begin{enumerate}
\item For any sequential design $\left(X_{n}\right)_{n\geq1}$ the sample
paths of 
\[
J_{n}:\mathbb{X}\rightarrow\left[0,\infty\right),\:J_{n}\left(x\right)=\mathbb{E}_{n}\left[\mathcal{H}\left(\text{Cond}{}_{x,Z(x)}\left(P_{n}^{\xi}\right)\right)\right]
\]
 are continuous on the random sets
\[
C_{n,k}:=\left\{ x\in\mathbb{X}:\text{rank}\left(\Sigma_{n}\left(x\right)\right)=k\right\} \subseteq\b X
\]
 for $n\in\mathbb{N}$ and $k=0,...,d$.
\item Assume that the covariance function $k$ of the underlying Gaussian
process $\xi$ is positive definite and $\c T\left(x\right)=\tau\left(x\right)\tau\left(x\right)^{\top}$
is positive definite for all $x\in\b X$. Then there exists a SUR
sequential design $\left(X_{n}\right)_{n\geq1}$ associated with $\c H$. 
\item There exists an $\epsilon$-quasi SUR sequential design $\left(X_{n}\right)_{n\geq1}$
associated with $\c H$. 
\end{enumerate}
\end{lem}
\begin{defn}
Let $\mathcal{H}$ be an uncertainty functional on $\mathbb{M}$ that
has the supermartingale property.
\begin{enumerate}
\item The expected gain functional at $x\in\mathbb{X}$ is defined
by 
\[
\mathcal{G}_{x}:\mathbb{M}\rightarrow\left[0,\infty\right),\,\mathcal{G}_{x}\left(\nu\right):=\mathcal{H}\left(\nu\right)-\mathcal{J}_{x}\left(\nu\right).
\]
\item The maximal expected gain functional is defined by 
\[
\mathcal{G}:\mathbb{M}\rightarrow\left[0,\infty\right),\,\mathcal{G}\left(\nu\right):=\sup_{x\in\mathbb{X}}\mathcal{G}_{x}\left(\nu\right).
\]
\end{enumerate}
\end{defn}

\section{Consistency of multivariate excursion set estimation under SUR sequential
design\label{sec:5}}

In the previous section we have recalled some desirable properties
of a SUR sequential design strategy that guarantee existence and continuity
of the sample criterion on a partition of the domain $\mathbb{X}$.
The following Proposition is the key to proving consistency of SUR
sequential design in the case of multivariate excursion set estimation.
The Proposition (see proof in Appendix \ref{subsec:app:sec5}) follows
from two Theorems in \cite{key-1} that are also stated for completeness
in Appendix \ref{subsec:app:sec5}. The proofs of the Theorems have
been adjusted to the multivariate setting.
\begin{prop}
\label{prop:7}Let $\mathcal{H}$ be an uncertainty functional on
$\mathbb{M}$, $\left(X_{n}\right)_{n\geq1}$ be an $\varepsilon$-quasi
SUR sequential design for $\mathcal{H}$ and $\mathcal{G}$ the associated
maximal expected gain functional. Assume that
\begin{enumerate}
\item $\mathcal{H}$ is $\mathfrak{P}$-continuous, $\mathfrak{P}$-uniformly
integrable and has the supermartingale property,
\item $\left\{ \nu\in\mathbb{M}:\mathcal{H}\left(\nu\right)=0\right\}
= \left\{ \nu\in\mathbb{M}:\mathcal{G}\left(\nu\right)=0\right\}.$ 
\end{enumerate}
Then it holds 
\[
H_{n}:=\mathcal{H}\left(P_{n}^{\xi}\right)\xrightarrow[n\rightarrow\infty]{a.s.}0.
\]

\end{prop}

\subsection{Integrated Bernoulli Variance (IBV)}
\label{subsec:IBV}

In this subsection we are turning to the integrated Bernoulli variance,
that is for example used for uncertainty reduction in \cite{key-3}
in the case of river plume mapping as already mentioned in the Introduction.
The proofs are inspired by \cite{key-1} and can be found in Appendix
\ref{subsec:app:sec5}.

Let $\xi$ be a Gaussian random element in $\b S$. The residual uncertainty
of the integrated Bernoulli variance (IBV) is defined as the random
variable
\[
H_{n}^{IBV}:=\int_{\mathbb{X}}p_{n}\left(u\right)\left(1-p_{n}\left(u\right)\right)\mu\left(du\right)=\int_{\mathbb{X}}\text{Var}\left({\bf 1}_{\Gamma\left(\xi\right)}\left(u\right)|\c F_{n}\right)\mu\left(du\right),
\]
where $\mathcal{F}_{n}$ is the $\sigma$-algebra generated by $n\in\b N$
observations and
\[
p_{n}\left(u\right):=\mathbb{E}\left[{\bf 1}_{\Gamma\left(\xi\right)}\left(u\right)|\mathcal{F}_{n}\right]=P\left(\xi\left(u\right)\geq T|\mathcal{F}_{n}\right).
\]
More generally, we can define the corresponding uncertainty functional
$\mathcal{H}^{IBV}$ by the mapping 
\begin{align*}
\mathcal{H}^{IBV}:\mathbb{M} & \rightarrow\left[0,\infty\right)\\
\nu & \mapsto\int_{\mathbb{X}}p_{\nu}\left(u\right)\left(1-p_{\nu}\left(u\right)\right)\mu\left(du\right),
\end{align*}
where $p_{\nu}\left(u\right):=\int_{\mathbb{S}}{\bf 1}_{\Gamma\left(f\right)}\left(u\right)\nu\left(df\right)$.
Note that $\mathcal{H}^{IBV}$ is clearly an uncertainty functional
on $\mathbb{M}$. Furthermore, let $\mathcal{G}^{IBV}$ be the associated
maximal expected gain functional. 

We want to use Proposition \ref{prop:7} to show 
\[
H_{n}^{IBV}\xrightarrow[n\rightarrow\infty]{a.s.}0
\]
for any $\epsilon$-quasi SUR sequential design $\left(X_{n}\right)_{n\geq1}$
for $\mathcal{H}^{IBV}$. In the following we will check the assumptions
of the Proposition. 
\begin{lem}
\label{lem:4}$\mathcal{H}^{IBV}$ is $\mathfrak{P}$-uniformly integrable
and has the supermartingale property.
\end{lem}

Recall that for every sequence
$\left(\bm{\nu}_{n}\right)_{n\geq1}\subset\mathfrak{P}\left(\xi\right)$
such that $\bm{\nu}_{n}\xrightarrow{}\bm{\nu}_{\infty}\in\mathfrak{P}\left(\xi\right)$ almost surely, 
it holds 
\[
\mathcal{H}^{IBV}\left(\bm{\nu}_{n}\right)\xrightarrow[n\rightarrow\infty]{a.s.}\mathcal{H}^{IBV}\left(\bm{\nu}_{\infty}\right)
\]
and that $\bm{\nu}_{n}\in\mathfrak{P}\left(\xi\right)$ with $\xi\sim\mathcal{GP}_{d}\left(m,k\right)$
if there exist $\sigma$-algebras $\mathcal{G}_{n}\subset\mathcal{F}$
such that $\bm{\nu}_{n}=P\left(\xi\in\cdot|\mathcal{G}_{n}\right)$
and $\bm{\nu}_{n}\left(\omega\right)$ is a Gaussian measure on $\left(\mathcal{C}\left(\mathbb{X},\mathbb{R}^{d}\right),\left\Vert \cdot\right\Vert _{\infty}\right)$.
In this case we can write $\bm{\nu}_{n}=\mathcal{GP}_{d}\left(m_{n},k_{n}\right)$
for some random mean function $m_{n}$ and random covariance function
$k_{n}$. We have 
\[
\bm{\nu}_{n}\xrightarrow[n\rightarrow\infty]{a.s.}\bm{\nu}_{\infty}\in\mathfrak{P}\left(\xi\right),
\]
if 
\[
m_{n}\xrightarrow[n\rightarrow\infty]{a.s.}m_{\infty}
\]
uniformly on $\X$ and 
\[
k_{n}\xrightarrow[n\rightarrow\infty]{a.s.}k_{\infty}
\]
uniformly on $\XX$, where 
\[
\bm{\nu}_{\infty}=\mathcal{GP}_{d}\left(m_{\infty},k_{\infty}\right)=P\left(\xi\in\cdot|\mathcal{G}_{\infty}\right)
\]
for some $\mathcal{G}_{\infty}\subset\mathcal{F}$. Furthermore, we
can write
\[
\mathcal{H}^{IBV}\left(\bm{\nu}_{n}\right)=\int_{\mathbb{X}}g\left(P\left(\xi\left(u\right)\geq T|\mathcal{G}_{n}\right)\right)\mu\left(du\right)
\]
for the bounded continuous function $g:\left[0,1\right]\rightarrow\left[0,\frac{1}{2}\right],\,x\mapsto x\left(1-x\right)$,
so the claim follows by the Dominated Convergence Theorem if we can
show 
\[
P\left(\xi\left(u\right)\geq T|\mathcal{G}_{n}\right)=P\left(\xi_{1}\left(u\right)\geq \t_{1},...,\xi_{d}\left(u\right)\geq \t_{d}|\mathcal{G}_{n}\right)\xrightarrow[n\rightarrow\infty]{a.s.}P\left(\xi\left(u\right)\geq T|\mathcal{G}_{\infty}\right).
\]

We have for almost all $\omega\in\Omega$ and all $u\in\b X$ by definition
of the multivariate Gaussian process that $\xi\left(u\right)\sim\mathcal{N}_{d}\left(m\left(u\right),k\left(u,u\right)\right)$
and 
\[
\nu_{n}\left(u,\omega\right):=\mathcal{L}\left(\left(\xi\left(u\right)|\mathcal{G}_{n}\right)\left(\omega\right)\right)=\mathcal{N}_{d}\left(m_{n}\left(u\right)\left(\omega\right),k_{n}\left(u,u\right)\left(\omega\right)\right).
\]
for all $n\in\b N\cup\left\{ \infty\right\} $ with
\[
m_{n}\left(u\right)\left(\omega\right)\xrightarrow[n\rightarrow\infty]{}m_{\infty}\left(u\right)\left(\omega\right),
\]
\[
k_{n}\left(u,u\right)\left(\omega\right)\xrightarrow[n\rightarrow\infty]{}k_{\infty}\left(u,u\right)\left(\omega\right)
\]
by the almost sure uniform convergence of $m_{n}$ and $k_{n}$. This
already implies $\nu_{n}\left(u,\omega\right)\xrightarrow[n\rightarrow\infty]{w}\nu_{\infty}\left(u,\omega\right)$
for almost all $\omega\in\Omega$ and all $u\in\X$, so by the Portmanteau
Theorem
\begin{align*}
\nu_{n}\left(u,\omega\right)\left({\bf T}\right)= & P\left(\xi_{1}\left(u\right)\geq \t_{1},...,\xi_{d}\left(u\right)\geq \t_{d}|\mathcal{G}_{n}\right)\left(\omega\right)\\
\xrightarrow[n\rightarrow\infty]{} & P\left(\xi_{1}\left(u\right)\geq \t_{1},...,\xi_{d}\left(u\right)\geq \t_{d}|\mathcal{G}_{\infty}\right)\left(\omega\right)\\
= & \nu_{\infty}\left(u,\omega\right)\left({\bf T}\right),
\end{align*}
if
\[
\nu_{\infty}\left(u,\omega\right)\left(\partial{\bf T}\right)=P\left(\exists i\in\left\{ 1,...,d\right\} :\xi_{i}\left(u\right)=\t_{i}|\mathcal{G}_{\infty}\right)\left(\omega\right)=0,
\]
which clearly holds if $\forall j\in\left\{ 1,...,d\right\} :\text{ }k_{\infty}\left(u,u\right)\left(\omega\right)_{jj}>0$
or $m_{\infty}\left(u\right)\left(\omega\right)_{j}\neq \t_{j}$, but
turns out to be more difficult in the other cases. For the prove that $\mathcal{H}^{IBV}$ is $\mathfrak{P}$-continuous, we need to construct a suitable finite decomposition of $\b X$ to check the convergence in each case. 
\begin{lem}
\label{lem:13}
Define the functions $F_1:\mathbb{X} \times \mathcal{P}(\{1,...,d\}) \rightarrow [0,\infty)$ and $F_2:\mathbb{X} \times \mathcal{P}(\{1,...,d\}) \times \Omega \rightarrow [0,\infty)$by
\[
F_1(u,J) = \sum_{j\in J}k\left(u,u\right)_{jj}^{2}
\]
and
\[
F_2(u,J,\omega) = \sum_{j\in J}\left(m_{\infty}\left(u\right)\left(\omega\right)_{j}-\t_{j}\right)^{2} + k_{\infty}\left(u,u\right)\left(\text{\ensuremath{\omega}}\right)_{jj}^{2}.
\]
For $J_1,J_2\subseteq\left\{ 1,...,d\right\}$ and $\omega \in \Omega$ fixed let $B_{J_1,J_2}(\omega) \subseteq \mathbb{X}$ be the set of all $u \in \mathbb{X}$ such that 
\begin{enumerate}
    \item $F_1(u,J_1)=0$
    \item $F_2(u,J_2, \omega)=0$
    \item For every $J'_1 \supset J_1$ and $J'_2 \supset J_2$ it holds $F_1(u,J'_1)>0$ and $F_2(u,J'_2,\omega)>0$.
\end{enumerate}
Then $\mathbb{X}$ can be written as the disjoint union 
\[
\bigcup_{J_1,J_2\subseteq\left\{ 1,...,d\right\} }B_{J_1,J_2}\left(\omega\right)
\]
and it holds:
\begin{enumerate}
\item If $J_2\not\subseteq J_1$, then $B_{J_1,J_2}$ is 
almost surely a $\mu$- null set.
\item If $J_2\subseteq J_1$, then
\[
P\left(\xi\left(u\right)\geq T|\mathcal{G}_{n}\right)\xrightarrow[n\rightarrow\infty]{a.s.}P\left(\xi\left(u\right)\geq T|\mathcal{G}_{\infty}\right)
\]
for all $u \in B_{J_1,J_2}$.
\end{enumerate}

\end{lem}
\begin{lem}
\label{lem:5}$\mathcal{H}^{IBV}$ is $\mathfrak{P}$-continuous.
\end{lem}
\begin{proof}
Using the decomposition for $\mathbb{X}$ from the above Lemma \ref{lem:13} and recalling that the finite union of $P$-null sets is again a $P$-null set,
we get by the first property in Lemma \ref{lem:13}
\[
\mu\left(\mathbb{X}\right)\overset{a.s.}{=}\mu\left(A\right),
\]
where the random subset $A$ is defined by 
\[
A\left(\omega\right):=\bigcup_{\underset{J_2\subseteq J_1}{J_1,J_2\subseteq\left\{ 1,...,d\right\} }}B_{J_1,J_2}\left(\omega\right).
\]
Hence we can conclude with $g:\left[0,1\right]\rightarrow\left[0,\frac{1}{2}\right],\,x\mapsto x\left(1-x\right)$ that
\begin{align*}
 & \int_{\mathbb{X}}g\left(P\left(\xi\left(u\right)\geq T|\mathcal{G}_{n}\right)\right)\mu\left(du\right)\\
\overset{a.s.}{=} & \int_{A}g\left(P\left(\xi\left(u\right)\geq T|\mathcal{G}_{n}\right)\right)\mu\left(du\right)\\
\xrightarrow[n\rightarrow\infty]{a.s.} & \int_{A}g\left(P\left(\xi\left(u\right)\geq T|\mathcal{G}_{\infty}\right)\right)\mu\left(du\right)\\
\overset{a.s.}{=} & \int_{\mathbb{X}}g\left(P\left(\xi\left(u\right)\geq T|\mathcal{G}_{\infty}\right)\right)\mu\left(du\right).
\end{align*}
by the second property in Lemma \ref{lem:13} and dominated convergence.
\end{proof}
\begin{lem}
\label{lem:6}
\begin{align*}
\left\{ \nu\in\mathbb{M}:\mathcal{H}^{IBV}\left(\nu\right)=0\right\}
= \left\{ \nu\in\mathbb{M}:\mathcal{G}^{IBV}\left(\nu\right)=0\right\}.
\end{align*}
\end{lem}
Note that 
$\left\{ \nu\in\mathbb{M}:\mathcal{H}\left(\nu\right)=0\right\}
\subseteq
\left\{ \nu\in\mathbb{M}:\mathcal{G}\left(\nu\right)=0\right\}$
always holds as shown in \cite{key-1}. In the proof of the above Lemma (see Appendix) we will only focus
on the reverse inclusion. 

\begin{thm}
\label{thm:9}
If $\left(X_{n}\right)_{n\geq1}$ is an $\varepsilon$-quasi SUR sequential
design for $\mathcal{H}^{IBV}$, then it holds 
\[
H_{n}^{IBV}:=\mathcal{H}^{IBV}\left(P_{n}^{\xi}\right)\xrightarrow[n\rightarrow\infty]{a.s.}0.
\] 
Furthermore, it holds almost surely and in
$L^{1}$ that,
\[
\int_{\mathbb{X}}\left({\bf 1}_{\xi\left(u\right)\geq T}-p_{n}\left(u\right)\right)^{2}\mu\left(du\right)\xrightarrow[n\rightarrow\infty]{}0.
\]
\end{thm}

\subsection{Excursion Measure Variance (EMV)}

Let $\xi$ be a Gaussian random element in $\b S$. Another popular
measure of the residual uncertainty with respect to the excursion
set that is used in \cite{key-1}, is the variance of the excursion
volume or excursion measure variance (EMV) defined by
\[
H_{n}^{EMV}:=\b E\left[\left(\alpha\left(\xi\right)-\b E\left[\alpha\left(\xi\right)|\c F_{n}\right]\right)^{2}|\c F_{n}\right]=\text{Var}\left(\alpha\left(\xi\right)|\c F_{n}\right),
\]
where $\mathcal{F}_{n}$ is the $\sigma$-algebra generated by $n\in\b N$
observations and $\alpha\left(\xi\right):=\mu\left(\Gamma\left(\xi\right)\right)$.
More generally, we can in this case define the corresponding uncertainty
functional $\mathcal{H}^{EMV}$ by the mapping 
\begin{align*}
\mathcal{H}^{EMV}:\mathbb{M} & \rightarrow\left[0,\infty\right)\\
\nu & \mapsto\int_{\mathbb{S}}\left(\alpha\left(f\right)-\bar{\alpha}_{\nu}\right)^{2}\nu\left(df\right),
\end{align*}
where $\bar{\alpha}_{\nu}:=\int_{\mathbb{S}}\alpha\left(f\right)\nu\left(df\right)$.
Note that $\mathcal{H}^{EMV}$ is clearly an uncertainty functional
on $\mathbb{M}$. Furthermore, let $\mathcal{G}^{EMV}$ be the associated
maximal expected gain functional. 

We want again to use Proposition \ref{prop:7} to show 
\[
H_{n}^{EMV}\xrightarrow[n\rightarrow\infty]{a.s.}0
\]
for any $\epsilon$-quasi SUR sequential design $\left(X_{n}\right)_{n\geq1}$
for $\mathcal{H}^{EMV}$. The ideas for the proofs are again based
on results shown in \cite{key-1} and can be found in Appendix \ref{subsec:app:sec5}.
\begin{lem}
\label{lem:7}$\mathcal{H}^{EMV}$ is $\mathfrak{P}$-uniformly integrable
and has the supermartingale property.
\end{lem}
Using Fubini's Theorem one can see that for the $\mathfrak{P}$-continuity of $\mathcal{H}^{EMV}$ it is necessary to deal with the covariance of ${\bf 1}_{\Gamma\left(\xi\right)}\left(u_1\right)$ and ${\bf 1}_{\Gamma\left(\xi\right)}\left(u_2\right)$ at two points $u_1,u_2 \in \mathbb{X}$. We need a similar result as 2. in \ref{lem:13} that is given by the next Lemma.
\begin{lem}
\label{lem:14}
For $J_{2}^{i}\subseteq J_{1}^{i}\subseteq\{1,...,d\}$ with $i=1,2$ it holds
\begin{align*}
 \text{Cov}\left({\bf 1}_{\Gamma\left(\xi\right)}\left(u_{1}\right),{\bf 1}_{\Gamma\left(\xi\right)}\left(u_{2}\right)|\c G_{n}\right)
\xrightarrow[n\rightarrow\infty]{a.s.} 
\text{Cov}\left({\bf 1}_{\Gamma\left(\xi\right)}\left(u_{1}\right),{\bf 1}_{\Gamma\left(\xi\right)}\left(u_{2}\right)|\c G_{\infty}\right)
\end{align*}
for all $u_1 \in B_{J_{1}^{1},J_{2}^{1}}$ and $u_2 \in B_{J_{1}^{2},J_{2}^{2}}$
\end{lem}
\begin{lem}
\label{lem:8}$\mathcal{H}^{EMV}$ is $\mathfrak{P}$-continuous. 
\end{lem}
\begin{proof}
Using Fubini's Theorem we have
\begin{align*}
\mathcal{H}^{EMV}\left(\bm{\nu}_{n}\right)
= & \b E\left[\left(\alpha\left(\xi\right)-\b E\left[\alpha\left(\xi\right)|\c G_{n}\right]\right)^{2}|\c G_{n}\right]\\
= & \b E\left[\left(\int_{\b X}{\bf 1}_{\Gamma\left(\xi\right)}\left(u\right)\mu\left(du\right)\right)^{2}|\c G_{n}\right]-\b E\left[\int_{\b X}{\bf 1}_{\Gamma\left(\xi\right)}\left(u\right)\mu\left(du\right)|\c G_{n}\right]^{2}\\
= & \int_{\b X}\int_{\b X}\b E\left[{\bf 1}_{\Gamma\left(\xi\right)}\left(u_{1}\right){\bf 1}_{\Gamma\left(\xi\right)}\left(u_{2}\right)|\c G_{n}\right]\mu\left(du_{1}\right)\mu\left(du_{2}\right)\\
 & -\int_{\b X}\b E\left[{\bf 1}_{\Gamma\left(\xi\right)}\left(u_{1}\right)|\c G_{n}\right]\mu\left(du_{1}\right)\int_{\b X}\b E\left[{\bf 1}_{\Gamma\left(\xi\right)}\left(u_{2}\right)|\c G_{n}\right]\mu\left(du_{2}\right)\\
= & \int_{\b X}\int_{\b X}\text{Cov}\left({\bf 1}_{\Gamma\left(\xi\right)}\left(u_{1}\right),{\bf 1}_{\Gamma\left(\xi\right)}\left(u_{2}\right)|\c G_{n}\right)\mu\left(du_{1}\right)\mu\left(du_{2}\right).
\end{align*}
We can now use the same decomposition $\b X=\bigcup_{J_{1},J_{2}\subseteq\left\{ 1,...,d\right\} }B_{J_{1},J_{2}}$
as in Lemma \ref{lem:13} and already know that $B_{J_{1},J_{2}}$
is almost surely a $\mu$- null set if $J_{2}\not\subseteq J_{1}$.
Hence the claim follows by Lemma \ref{lem:14} and the Dominated Convergence as
\begin{align*}
 & \int_{B_{J_{1}^{2},J_{2}^{2}}\left(\omega\right)}\int_{B_{J_{1}^{1},J_{2}^{1}}\left(\omega\right)}\text{Cov}\left({\bf 1}_{\Gamma\left(\xi\right)}\left(u_{1}\right),{\bf 1}_{\Gamma\left(\xi\right)}\left(u_{2}\right)|\c G_{n}\right)\left(\omega\right)\mu\left(du_{1}\right)\mu\left(du_{2}\right)\\
\xrightarrow[n\rightarrow\infty]{} & \int_{B_{J_{1}^{2},J_{2}^{2}}\left(\omega\right)}\int_{B_{J_{1}^{1},J_{2}^{1}}\left(\omega\right)}\text{Cov}\left({\bf 1}_{\Gamma\left(\xi\right)}\left(u_{1}\right),{\bf 1}_{\Gamma\left(\xi\right)}\left(u_{2}\right)|\c G_{\infty}\right)\left(\omega\right)\mu\left(du_{1}\right)\mu\left(du_{2}\right)
\end{align*}
for almost all $\omega\in\Omega$.
\end{proof}

To apply Proposition \ref{prop:7} it remains to show that $\mathcal{H}^{EMV}$ and $\mathcal{G}^{EMV}$ vanish on the same subset of $\mathbb{M}$. It can be deduced by the same steps as in part (f) in the proof of Theorem 4.3 in \cite{key-1} that $\alpha\left(\xi\right)-\b E\text{\ensuremath{\left[\alpha\left(\xi\right)\right]}}$ is orthogonal to $L^{2}\left(\Omega,\sigma\left(Z\left(x\right)\right),P\right)$ for all $x\in \mathbb{X}$, where $Z\left(x\right)=\xi\left(x\right)+\tau\left(x\right)U$, $U\sim\c N_{d}\left(0,I_{d}\right)$
independent of $\xi$, since also for a multivariate Gaussian process $\xi$ we have that $\alpha(\xi)$ is only a random variable. The bottleneck is to conclude that $\alpha\left(\xi\right)-\b E\text{\ensuremath{\left[\alpha\left(\xi\right)\right]}}$
is also orthogonal to $L^{2}\left(\Omega,\sigma\left(\xi\left(x\right)\right),P\right)$, which can be handled by the following Lemma. 
\begin{lem}
\label{lem:aux}
Let $V=\left(V_{1},...,V_{d}\right)$ and $W=\left(W_{1},...,W_{d}\right)$
be independent Gaussian random vectors in $\left(\b R^{d},\c B\left(\b R^{d}\right)\right)$
and $U$ be a random variable in $\left(\b R,\c B\left(\b R\right)\right)$,
all defined on the probability space $\left(\Omega,\mathcal{F},P\right)$.
Assume that $W$ is independent of $\left(U,V\right)$ and that $U$
is orthogonal to $L^{2}\left(\Omega,\sigma\left(V+W\right),P\right)$,
that means for every $\sigma\left(V+W\right)$-measurable and square-integrable
random variable $X:\Omega\rightarrow\b R$ we have $\mathbb{E}\left[UX\right]=0$.
Then is $U$ also orthogonal to $L^{2}\left(\Omega,\sigma\left(V\right),P\right)$.
\end{lem}
\begin{lem}
\label{lem:9}
\begin{align*}
\left\{ \nu\in\mathbb{M}:\mathcal{H}^{EMV}\left(\nu\right)=0\right\}
=  \left\{\nu\in\mathbb{M}:\mathcal{G}^{EMV}\left(\nu\right)=0\right\}. 
\end{align*}
\end{lem}
Combining Lemmas \ref{lem:7}, \ref{lem:8} and \ref{lem:9} we get by means of Proposition \ref{prop:7} and the same martingale convergence arguments for $\left(\mathbb{E}\left[\alpha\left(\xi\right)|\c F_{n}\right]\right)_{n\in \mathbb{N}}$ as in the proof of Proposition 4.5 in \cite{key-1} 
the following Theorem.
\begin{thm}
If $\left(X_{n}\right)_{n\geq1}$ is an $\varepsilon$-quasi SUR sequential
design for $\mathcal{H}^{EMV}$, then it holds
\[
H_{n}^{EMV}:=\mathcal{H}^{EMV}\left(P_{n}^{\xi}\right)\xrightarrow[n\rightarrow\infty]{a.s.}0.
\]
Furthermore, we have almost surely and in
$L^{1}$ that,
\[
\mathbb{E}\left[\alpha\left(\xi\right)|\c F_{n}\right]\xrightarrow[n\rightarrow\infty]{}\alpha\left(\xi\right).
\]
\end{thm}

\section{Conclusion}

We have successfully extended the consistency results for the SUR
sequential design strategies based on the integrated Bernoulli variance
functional (IBV) and the variance of the excursion volume functional
(EMV), that address the estimation of the excursion set problem, as
introduced and proven for the univariate setting in \cite{key-1}
to the multivariate setting based on multivariate Gaussian processes
$\xi=\left(\xi_{1},...,\xi_{d}\right)$ with sample paths in the function
space $\c C\left(\b X;\mathbb{R}^{d}\right)$. 

Bect et al. have furthermore proven consistency for the knowledge
gradient functional and the expected improvement functional. However,
multi-objective optimization with multivariate Gaussian processes
is beyond the scope of this paper and not considered. Nevertheless
can our results, i.e. from Section \ref{sec:2}, be of interest for
further research in this area. 

The multivariate setting for excursion set estimation also arises
in \cite{key-3} and our work provides a (slightly relaxed) theoretical
foundation for the techniques that are used in the paper for river
plume mapping. Note that the excursion sets that we are addressing
have the special form 
\[
\Gamma\left(\xi\right)=\left\{ u\in\mathbb{X}:\xi\left(u\right)\geq T\right\} ,
\]
due to the orthants ${\bf T}:=\left[\t_{1},\infty\right)\times...\times\left[\t_{d},\infty\right)$
that we are considering. It remains to be checked if the general case
of arbitrary closed sets ${\bf T}\subseteq\mathbb{R}^{d}$ also holds.

Further studies should also include the convergence rate of the SUR
sequential design to provide a full theoretical support for their
effectiveness. An important question in this context is also if the
correlation (similarity) of the Gaussian processes $\left(\xi_{1},...,\xi_{d}\right)$
has an enhancing effect on the convergence rate.

\appendix
\section{Proofs and Auxiliary Results}

\subsection{Proofs of Section 2 and Auxiliary Results}

\begin{lem}
\label{lem:10}Let $\left(\b X,\mathfrak{d}\right)$ be a compact metric space.
The product space $\mathcal{C}\left(\mathbb{X}\right)\times...\times\mathcal{C}\left(\mathbb{X}\right)$
and $\mathcal{C}\left(\mathbb{X};\mathbb{R}^{d}\right)$ with the
supremum norm 
\[
\left\Vert f\right\Vert _{\infty}=\sup_{x\in\mathbb{X}}\left\Vert f\left(x\right)\right\Vert _{\max}
\]
 are isomorphic and have the same topological structure. They even
form an isometric isomorphism if we consider the product space with
the norm 
\[
\left\Vert \left(f_{1},...,f_{d}\right)\right\Vert _{d,\infty}=\max_{i\in\left\{ 1,...,d\right\} }\left\Vert f_{i}\right\Vert _{\infty}.
\]
\end{lem}
\begin{proof}
Define the mapping $L:\mathcal{C}\left(\mathbb{X}\right)\times...\times\mathcal{C}\left(\mathbb{X}\right)\rightarrow\mathcal{C}\left(\mathbb{X};\mathbb{R}^{d}\right)$
for $x\in\b X$ by
\[
L\left(\left(f_{1},...,f_{d}\right)\right)\left(x\right)=\left(f_{1}\left(x\right),...,f_{d}\left(x\right)\right)\in\mathbb{R}^{d}.
\]
That $L$ is well-defined follows by the universal property of the
product space $\b R^{d}\cong\b R\times...\times\b R$. Indeed, the
universal property states that the product topology on $\b R^{d}$
provides for continuous function $f_{i}:\b X\rightarrow\b R$, $i\in\left\{ 1,...,d\right\} $,
the existence of a unique continuous function $f:\b X\rightarrow\b R^{d}$
such that $f_{i}=\pi_{i}\circ f$.

It is easy to see that this map is injective since $L\left(\left(f_{1},...,f_{d}\right)\right)=L\left(\left(g_{1},...,g_{d}\right)\right)$
implies $f_{i}\left(x\right)=g_{i}\left(x\right)$ for all $x\in\b X$
and $i\in\left\{ 1,...,d\right\} $. By continuity we conclude $f_{i}=g_{i}$
for all $i\in\left\{ 1,...,d\right\} $. For surjectivity let $f\in\mathcal{C}\left(\mathbb{X};\mathbb{R}^{d}\right)$,
then we have $f=L\left(\left(\pi_{1}\left(f\right),...,\pi_{d}\left(f\right)\right)\right)$,
where $\pi_{i}\left(f\right)\left(x\right)=f\left(x\right)_{i}$ is
the continuous projection map on the i-th component of $f$. Hence
the spaces $\mathcal{C}\left(\mathbb{X}\right)\times...\times\mathcal{C}\left(\mathbb{X}\right)$
and $\mathcal{C}\left(\mathbb{X};\mathbb{R}^{d}\right)$ are isomorphic. 

We furthermore have 
\begin{align*}
\left\Vert \left(f_{1},...,f_{d}\right)\right\Vert _{d,\infty} & =\max_{i\in\left\{ 1,...,d\right\} }\left\Vert f_{i}\right\Vert _{\infty}\\
 & =\max_{i\in\left\{ 1,...,d\right\} }\left(\sup_{x\in\b X}\left|f_{i}\left(x\right)\right|\right)\\
 & =\sup_{x\in\mathbb{X}}\left(\max_{i\in\left\{ 1,...,d\right\} }\left|f_{i}\left(x\right)\right|\right)\\
 & =\sup_{x\in\mathbb{X}}\left\Vert L\left(f_{1},...,f_{d}\right)\left(x\right)\right\Vert _{\max}\\
 & =\left\Vert L\left(f_{1},...,f_{d}\right)\right\Vert _{\infty},
\end{align*}
which gives us the isometry. Since equivalent norms induce the same
topology the claim follows.
\end{proof}
The continuous dual space of a real normed vector space $\left(X,\left\Vert \cdot\right\Vert \right)$
is defined by 
\[
X^{*}:=\left\{ L:X\rightarrow\mathbb{R}:L\text{ is linear and continuous}\right\} .
\]
Recall that the operator norm $\left\Vert \cdot\right\Vert _{op}$
on the dual space $X^{*}$ is given by 
\[
\left\Vert L\right\Vert _{op}:=\inf\left\{ c\geq0:\left|Lx\right|\leq c\left\Vert x\right\Vert \text{ for all }x\in X\right\} =\sup_{\left\Vert x\right\Vert \leq1}\left|Lx\right|
\]
and that $L$ is bounded if and only if it is continuous. 
\begin{thm}
\label{thm:6}$\left(X_{1},\left\Vert \cdot\right\Vert _{X_{1}}\right),...,\left(X_{d},\left\Vert \cdot\right\Vert _{X_{d}}\right)$
be real Banach spaces with dual spaces $X_{1}^{*},...,X_{d}^{*}$.
Define the space 
\[
X_{1}\times...\times X_{d}:=\left\{ \left(x_{1},...,x_{d}\right):x_{i}\in X_{i},i\in\left\{ 1,...,d\right\} \right\} 
\]
 with norm $\left\Vert \left(x_{1},...,x_{d}\right)\right\Vert =\max_{i\in\left\{ 1,...,d\right\} }\left\Vert x_{i}\right\Vert _{X_{i}}$
and 
\[
X_{1}^{*}\times...\times X_{d}^{*}:=\left\{ \left(L_{1},...,L_{d}\right):L_{i}\in X_{i}^{*},i\in\left\{ 1,...,d\right\} \right\} 
\]
 with norm $\left\Vert \left(L_{1},...,L_{d}\right)\right\Vert _{*}=\sum_{i=1}^{d}\left\Vert L_{i}\right\Vert _{op,X_{i}}$.
Then the following statements hold.
\begin{enumerate}
\item $\left(X_{1}\times...\times X_{d},\left\Vert \cdot\right\Vert \right)$
and $\left(X_{1}^{*}\times...\times X_{d}^{*},\left\Vert \cdot\right\Vert _{*}\right)$
are Banach spaces.
\item $J:X_{1}^{*}\times...\times X_{d}^{*}\rightarrow\left(X_{1}\times...\times X_{d}\right)^{*}$
defined by 
\[
J\left(L_{1},...,L_{d}\right)\left(x_{1},...,x_{d}\right)=\sum_{i=1}^{d}L_{i}x_{i}
\]
is an isometric isomorphism.
\end{enumerate}
\end{thm}
\begin{proof}
Since we are considering finite products, $X_{1}\times...\times X_{d}$
is again a Banach space with the norm $\left\Vert \left(x_{1},...,x_{d}\right)\right\Vert =\max_{i\in\left\{ 1,...,d\right\} }\left\Vert x_{i}\right\Vert _{X_{i}}$
that induces the product topology. That $X_{1}^{*}\times...\times X_{d}^{*}$
is a Banach space follows by the same statement, since the (continuous)
dual space $X_{i}^{*}$ of the Banach space $X_{i}$ is again a Banach
space with the operator norm $\left\Vert \cdot\right\Vert _{op,X_{i}}$.
$X_{1}^{*}\times...\times X_{d}^{*}$ is indeed a product space, since
the product topology is also induced by the norm $\left\Vert \left(L_{1},...,L_{d}\right)\right\Vert _{*}=\sum_{i=1}^{d}\left\Vert L_{i}\right\Vert _{op,X_{i}}$.

The function $J$ is clearly linear. We continue by showing that it
is isometric with $\left\Vert J\left(L_{1},...,L_{d}\right)\right\Vert _{op}=\left\Vert \left(L_{1},...,L_{d}\right)\right\Vert _{*}$.
Indeed, we have 
\begin{align*}
\left\Vert J\left(L_{1},...,L_{d}\right)\right\Vert _{op} & =\sup_{\left\Vert \left(x_{1,...,}x_{d}\right)\right\Vert =1}\left|J\left(L_{1},...,L_{d}\right)\left(x_{1},...,x_{d}\right)\right|\\
 & =\sup_{\left\Vert \left(x_{1,...,}x_{d}\right)\right\Vert =1}\left|\sum_{i=1}^{d}L_{i}x_{i}\right|\\
 & \leq\sup_{\left\Vert \left(x_{1,...,}x_{d}\right)\right\Vert =1}\sum_{i=1}^{d}\left\Vert L_{i}\right\Vert _{op,X_{i}}\left\Vert x_{i}\right\Vert _{X_{i}}\\
 & =\sum_{i=1}^{d}\left\Vert L_{i}\right\Vert _{op,X_{i}}\\
 & =\left\Vert \left(L_{1},...,L_{d}\right)\right\Vert _{*}.
\end{align*}
On the other hand it is easy to see that
\begin{align*}
\left\Vert J\left(L_{1},...,L_{d}\right)\right\Vert _{op} & =\sup_{\left\Vert \left(x_{1,...,}x_{d}\right)\right\Vert =1}\left|\sum_{i=1}^{d}L_{i}x_{i}\right|\\
 & \geq\sum_{i=1}^{d}\sup_{\left\Vert x_{i}\right\Vert =1}L_{i}x_{i}\\
 & =\sum_{i=1}^{d}\left\Vert L_{i}\right\Vert _{op,X_{i}},
\end{align*}
where the last equality follows by 
\[
\left\Vert L_{i}\right\Vert _{op,X_{i}}=\sup_{\left\Vert x_{i}\right\Vert =1}\left|L_{i}x_{i}\right|=\sup_{\left\Vert x_{i}\right\Vert =1}L_{i}x_{i},
\]
since $X_{i}$ is a real vector space and $L_{i}:X_{i}\rightarrow\b R$
is linear. Hence we conclude $\left\Vert J\left(L_{1},...,L_{d}\right)\right\Vert _{op}=\left\Vert \left(L_{1},...,L_{d}\right)\right\Vert _{*}$.

Since $J$ is linear and isometric we already know that it is injective.
It remains to show that it is also surjective. Note that we can write
for every element $\left(x_{1},...,x_{d}\right)\in X_{1}\times...\times X_{d}$
\begin{align*}
\left(x_{1},...,x_{d}\right) & =\left(x_{1},0,...,0\right)+...+\left(0,...,x_{d}\right)\\
 & =\sum_{k=1}^{d}i_{k}\left(x_{k}\right),
\end{align*}
where $i_{k}:X_{k}\rightarrow X_{1}\times...\times X_{d}$, $i_{k}\left(x_{k}\right)=\left(0,...,0,x_{k},0,...,0\right)$
are the continuous inclusion maps for $k\in\left\{ 1,...,d\right\} $.
This means for every $L\in\left(X_{1}\times...\times X_{d}\right)^{*}$
we have $L\circ i_{k}\in X_{k}^{*}$ and we can define $\left(L\circ i_{1},...,L\circ i_{d}\right)\in X_{1}^{*}\times...\times X_{d}^{*}$.
By linearity of $L$ it follows 
\begin{align*}
J\left(L\circ i_{1},...,L\circ i_{d}\right)\left(x_{1},...,x_{d}\right) & =\sum_{k=1}^{d}\left(L\circ i_{k}\right)\left(x_{k}\right)\\
 & =L\left(\sum_{k=1}^{d}i_{k}\left(x_{k}\right)\right)\\
 & =L\left(x_{1},...,x_{d}\right)
\end{align*}
for every $\left(x_{1},...,x_{d}\right)\in X_{1}\times...\times X_{d}$
and hence the surjectivity of $J$.
\end{proof}
\begin{proof}
\textbf{(Lemma \ref{thm:1})}\\
We first show that $\xi$ is a stochastic process with continuous
sample paths if and only if it is a random element in $\left(\mathcal{C}\left(\mathbb{X};\mathbb{R}^{d}\right),\left\Vert \cdot\right\Vert _{\infty}\right)$
with respect to its Borel $\sigma$-algebra. Assume that $\xi$ is
a stochastic process with continuous sample paths. $\xi$ is measurable
with respect to the $\sigma$-algebra $\mathcal{C}\left(\mathbb{X};\mathbb{R}^{d}\right)\cap\mathcal{B}\left(\mathbb{R}^{d}\right)^{\b X}$
by Lemma 4.1 in \cite{key-9}. Since $\mathcal{B}\left(\mathbb{R}^{d}\right)^{\b X}$
is generated by the projection (or evaluation) maps $\pi_{x}:\left(\mathbb{R}^{d}\right)^{\b X}\rightarrow\b R^{d}$
for $x\in\mathbb{X}$, we have that $\xi$ is also measurable with
respect to the Borel $\sigma$-algebra on $\mathcal{C}\left(\mathbb{X};\mathbb{R}^{d}\right)$.
This makes $\xi$ a random element in $\left(\mathcal{C}\left(\mathbb{X};\mathbb{R}^{d}\right),\left\Vert \cdot\right\Vert _{\infty}\right)$.
The other direction is clear since the evaluation maps are linear
and continuous.

Assume that $\xi$ is a multivariate Gaussian processes. By the above
part $\xi$ is a random element in $\mathcal{C}\left(\mathbb{X};\mathbb{R}^{d}\right)$
and it is Gaussian if $\left\langle \xi,L\right\rangle $ is a Gaussian
variable for all $L\in\mathcal{C}\left(\mathbb{X};\mathbb{R}^{d}\right)^{*}$.
By Proposition \ref{prop:2} there exist finite signed Borel measures
$\mu_{i}$ on $\left(\b X,d\right)$ such that
\[
\left\langle \xi,L\right\rangle =\sum_{i=1}^{d}\int_{\b X}\xi_{i}\left(x\right)\mu_{i}\left(dx\right).
\]
Let $D$ be a countable dense subset of $\mathbb{X}$. Since Borel
measures and Baire measures are equivalent on compact metric spaces
(see Chapter 7 and 8 in \cite{key-19}) and every finite measure $\mu_{i}$
is also a Radon measure since it is regular, there exists a sequence
of linear combinations of Dirac measures $\delta_{x_{k}^{\left(i\right)}}$
with $x_{k}^{\left(i\right)}\in D$ such that 
\[
\sum_{k=1}^{n}a_{k}^{\left(i\right)}\delta_{x_{k}^{\left(i\right)}}\xrightarrow[n\rightarrow\infty]{w}\mu_{i},
\]
where $a_{k}^{\left(i\right)}\in\mathbb{R}$, by Example 8.1.6 in
\cite{key-19} (see also Example 8.16 in \cite{key-37} or Chapter
15 in \cite{key-32}) for every $i\in\left\{ 1,...,d\right\} $. By
the definition of multivariate Gaussian processes we know that 
\[
\sum_{i=1}^{d}\sum_{k=1}^{n}a_{k}^{\left(i\right)}\xi_{i}\left(x_{k}^{\left(i\right)}\right)
\]
is a Gaussian variable for every $n\in\b N$ and 
\begin{align*}
    \sum_{i=1}^{d}\sum_{k=1}^{n}a_{k}^{\left(i\right)}\xi_{i}\left(x_{k}^{\left(i\right)}\right) =&\sum_{i=1}^{d}\int_{\b X}\xi_{i}\left(x\right)\left(\sum_{k=1}^{n}a_{k}^{\left(i\right)}\delta_{x_{k}^{\left(i\right)}}\right)\left(dx\right) \\ \xrightarrow[n\rightarrow\infty]{a.s.}& \sum_{i=1}^{d}\int_{\b X}\xi_{i}\left(x\right)\mu_{i}\left(dx\right).
\end{align*}
Since the almost sure limit of a sequence of Gaussian random variables
is again Gaussian, we conclude that $\left\langle \xi,L\right\rangle $
is Gaussian. (The moments are bounded since the sequence converges
also in distribution and hence the sequence is tight.)

Assume now that $\xi$ is a Gaussian random element. We know that
$\xi$ induces a Gaussian vector 
\[
\left(L_{1}\left(\xi\right),...,L_{d}\left(\xi\right)\right)
\]
for $L_{i}\in U_{i}$, where $U_{i}$ are subsets of the dual space
of $\mathcal{C}\left(\mathbb{X};\mathbb{R}^{d}\right)$ for $i\in\left\{ 1,...,d\right\} $.
Taking the evaluation maps $\delta_{x}:\mathcal{C}\left(\mathbb{X};\mathbb{R}^{d}\right)\rightarrow\mathbb{R}^{d},f\mapsto f\left(x\right)$
and the projection maps $\pi_{i}:\mathbb{R}^{d}\rightarrow\mathbb{R},x\mapsto x_{i}$
we have that $\pi_{i}\circ\delta_{x}:\mathcal{C}\left(\mathbb{X};\mathbb{R}^{d}\right)\rightarrow\b R$
is linear and continuous for every $i\in\left\{ 1,...,d\right\} $,
$x\in\mathbb{X}$. Hence we can define a $\R^{d}$-valued process
by $\left(\pi_{1}\circ\delta_{x}\left(\xi\right),...,\pi_{d}\circ\delta_{x}\left(\xi\right)\right)_{x\in\mathbb{X}}$,
whose finite-dimensional distributions are Gaussian.

The finite-dimensional distributions of $\xi$ are uniquely determined
by $m$ and $k$ and hence by Proposition 4.2 in \cite{key-9} the
last claim follows, since the $\sigma$-algebras coincide as already
mentioned in the first part of the proof.
\end{proof}
\begin{proof}
\textbf{(Lemma \ref{lem:2})}\\
Let $\xi=\left(\xi_{1},...,\xi_{d}\right)$ be a multivariate Gaussian
process with continuous sample paths. Then we have for $x\in\b X$
and any sequence $\left(x_{n}\right)_{n\in\b N}$ in $\b X$ with
$x_{n}\rightarrow x$ (with respect to the metric $d$ on $\b X$)
that $\sum_{i=1}^{d}\left(\xi_{i}\left(x_{n}\right)-\xi_{i}\left(x\right)\right)^{2}\xrightarrow{a.s.}0.$
Hence $\xi_{i}\left(x_{n}\right)$ is a Gaussian random variable for
every $i\in\left\{ 1,...,d\right\} $ and $n\in\b N$ with $\xi_{i}\left(x_{n}\right)\xrightarrow{a.s.}\xi_{i}\left(x\right)$,
so Lemma 1 in \cite{key-35} implies $\sum_{i=1}^{d}\b E\left[\left(\xi_{i}\left(x_{n}\right)-\xi_{i}\left(x\right)\right)^{2}\right]\xrightarrow{}0.$
This already implies continuity of the mean function $m$ since 
\[
\left\Vert m\left(x_{n}\right)-m\left(x\right)\right\Vert _{2}^{2}\leq\sum_{i=1}^{d}\b E\left[\left(\xi_{i}\left(x_{n}\right)-\xi_{i}\left(x\right)\right)^{2}\right]\xrightarrow{}0.
\]
For continuity of the covariance function let also $y\in\b X$ and
$\left(y_{n}\right)_{n\in\b N}$ be any sequence in $\b X$ with $y_{n}\rightarrow y$.
Since 
\[
k\left(x_{n},y_{n}\right)=\mathbb{E}\left[\xi(x_{n})\xi(y_{n})^{\top}\right]-m(x_{n})m(y_{n})^{\top},
\]
it only remains to show convergence of the first term on the right
hand side. This can be checked component-wise using Cauchy-Schwarz
and 
\[
\mathbb{E}\left[\left(\xi_{i}(x_{n})-\xi_{i}(x)\right)^{2}\right]\xrightarrow{}0
\]
for all $i\in\left\{ 1,...,d\right\} $.
\end{proof}

\subsection{Proofs of Section 3 \label{subsec:app:sec3}}
\begin{lem}
Let $\mathcal{C}\left(\mathbb{X};\mathbb{R}^{d}\right)$ and $\mathcal{C}\left(\mathbb{X}\times\mathbb{X};\mathbb{R}^{d\times d}\right)$
be endowed with their Borel $\sigma$-algebra. Define the mappings
\begin{align*}
m_{\bullet}:\mathbb{M}\rightarrow\mathcal{C}\left(\mathbb{X};\mathbb{R}^{d}\right), & \,\nu\mapsto m_{\nu}\\
k_{\bullet}:\mathbb{M}\rightarrow\mathcal{C}\left(\mathbb{X}\times\mathbb{X};\mathbb{R}^{d\times d}\right) & ,\,\nu\mapsto k_{\nu}\\
\Psi:=\left(m_{\bullet},k_{\bullet}\right)
\end{align*}
and let $\Theta\subset\mathcal{C}\left(\mathbb{X};\mathbb{R}^{d}\right)\times\mathcal{C}\left(\mathbb{X}\times\mathbb{X};\mathbb{R}^{d\times d}\right)$
be the range of $\Psi$ with trace $\sigma$-algebra $\Sigma_{\Theta}$
. Then $\Psi:\mathbb{M}\rightarrow\Theta$ is $\mathcal{M}/\Sigma_{\Theta}$-measurable
and its inverse $\Psi^{-1}:\Theta\rightarrow\mathbb{M}$ is $\Sigma_{\Theta}/\mathcal{M}$-measurable.
\end{lem}
\begin{proof}
We first concentrate on $m_{\bullet}:\mathbb{M}\rightarrow\mathcal{C}\left(\mathbb{X};\mathbb{R}^{d}\right)$: 

We can interpret $\mathcal{C}\left(\mathbb{X};\mathbb{R}^{d}\right)$
as the product space $\mathcal{C}\left(\mathbb{X}\right)\times...\times\mathcal{C}\left(\mathbb{X}\right)$
by Proposition \ref{lem:10}. Furthermore, it holds $\mathcal{B}\left(\mathcal{C}\left(\mathbb{X};\mathbb{R}^{d}\right)\right)=\mathcal{B}\left(\mathcal{C}\left(\mathbb{X}\right)\right)^{\otimes d}$,
so the mapping $m_{\bullet}=\left(m_{\bullet}^{(1)},...,m_{\bullet}^{(d)}\right)$
is $\mathcal{M}/\mathcal{B}\left(\mathcal{C}\left(\mathbb{X};\mathbb{R}^{d}\right)\right)$-measurable
if and only if $m_{\bullet}^{(i)}$ is $\mathcal{M}/\mathcal{B}\left(\mathcal{C}\left(\mathbb{X}\right)\right)$-measurable
for all $i=1,...,d$. This means it suffices to show that for all
functionals $L$ in the dual space of $\mathcal{C}\left(\mathbb{X}\right)$
the mapping 
\[
\nu\mapsto\left(L\circ m_{\bullet}^{(i)}\right)\left(\nu\right)=L\left(m_{\nu}^{(i)}\right)
\]
is $\mathcal{M}/\mathcal{B}\left(\mathbb{R}\right)$-measurable for
$i=1,...,d$, since weak and strong measurability coincides. For any
functional $L\in\mathcal{C}\left(\mathbb{X}\right)^{*}$ there exists
by the Riesz-Markov representation Theorem (compare Proposition \ref{prop:2})
a unique finite signed measure $\mu_{L}$ on $\mathbb{X}$ such that
$L\left(f\right)=\int_{\mathbb{X}}f\left(x\right)\mu_{L}\left(dx\right)$
for all $f\in\mathcal{C}\left(\mathbb{X}\right)$ and with Fubini's
Theorem we conclude 
\begin{align*}
L\left(m_{\nu}^{(i)}\right) & =\int_{\mathbb{X}}m_{\nu}\left(x\right)\mu_{L}\left(dx\right)\\
 & =\int_{\mathbb{X}}\left(\int_{\mathcal{C}\left(\mathbb{X}\right)}f\left(x\right)\nu\left(df\right)\right)\mu_{L}\left(dx\right)\\
 & =\int_{\mathcal{C}\left(\mathbb{X}\right)}\left(\int_{\mathbb{X}}f\left(x\right)\mu_{L}\left(dx\right)\right)\nu\left(df\right)\\
 & =\int_{\mathcal{C}\left(\mathbb{X}\right)}L\left(f\right)\nu\left(df\right).
\end{align*}
Hence the mapping $\nu\mapsto L\left(m_{\nu}^{(i)}\right)$ is measurable with respect to $\mathcal{M}/\mathcal{B}\left(\mathbb{R}\right)$
since $L$ is $\mathcal{B}\left(\mathcal{C}\left(\mathbb{X}\right)\right)/\mathbf{\mathcal{B}\left(\mathbb{R}\right)}$-measurable
as continuous linear functional. Indeed, for every $\mathcal{B}\left(\mathcal{C}\left(\mathbb{X}\right)\right)/\mathbf{\mathcal{B}\left(\mathbb{R}\right)}$-measurable
function $\varphi$ it is possible to show that 
\[
\nu\mapsto\int_{\mathcal{C}\left(\mathbb{X}\right)}\varphi\left(f\right)\nu\left(df\right)
\]
is $\mathcal{M}/\mathcal{B}\left(\mathbb{R}\right)$-measurable. This
follows first for every $\varphi$ of the form $\varphi=\boldsymbol{1}_{A}$
with $A\in\mathcal{B}\left(\mathcal{C}\left(\mathbb{X}\right)\right)$
since $\nu\mapsto\int_{\mathcal{C}\left(\mathbb{X}\right)}\varphi\left(f\right)\nu\left(df\right)=\pi_{A}\left(\nu\right)$
is $\mathcal{M}/\mathcal{B}\left(\mathbb{R}\right)$-measurable by
definition of $\c M$ and then extends by linearity of the integral
first to simple functions $\varphi$ and then to all $\mathcal{B}\left(\mathcal{C}\left(\mathbb{X}\right)\right)/\mathbf{\mathcal{B}\left(\mathbb{R}\right)}$-measurable
functions $\varphi$ by increasing approximation with simple functions.
This establishes that $m_{\bullet}^{(i)}$ is $\mathcal{M}/\mathcal{B}\left(\mathcal{C}\left(\mathbb{X}\right)\right)$-measurable
for all $i=1,...,d$ and hence also the measurability of $m_{\bullet}$.

A similar argument holds for $k_{\bullet}:\mathbb{M}\rightarrow\mathcal{C}\left(\mathbb{X}\times\mathbb{X};\mathbb{R}^{d\times d}\right)$.
There is an isometric isomorphism between $\mathbb{R}^{d\times d}$
and $\mathbb{R}^{d^{2}}$ given by 
\[
A=\left(a_{ij}\right)_{1\leq i,j\leq d}\mapsto\left(a_{11},...,a_{1d},a_{21},...,a_{dd}\right),
\]
so we can use the same steps as before to show that each component
is measurable with respect to $\mathcal{M}/\mathcal{B}\left(\mathcal{C}\left(\mathbb{X}\times\mathbb{X}\right)\right)$
and conclude with the same arguments.

That $\Psi:\mathbb{M}\rightarrow\Theta$ is $\mathcal{M}/\Sigma_{\Theta}$-measurable
follows easily by the measurability of $m_{\bullet}$ and $k_{\bullet}$.
Now $\Psi^{-1}$ is $\Sigma_{\Theta}/\mathcal{M}$-measurable if and
only if $\left(m,k\right)\mapsto\mathcal{GP}_{d}\left(m,k\right)\left(A\right)$
is measurable for all $A\text{\ensuremath{\in}}\mathcal{B}\left(\mathcal{C}\left(\mathbb{X};\mathbb{R}^{d}\right)\right)$.
The latter holds for all cylinder sets of the form $A=\bigcap_{k=1}^{n}\left\{ f\in\mathcal{C}\left(\mathbb{X};\mathbb{R}^{d}\right):f\left(x_{k}\right)\in\Gamma_{k}\right\} $
with $x_{k}\in\mathbb{X}$ and $\Gamma_{k}\in\mathcal{B}\left(\mathbb{R}^{d}\right)$
for $k=1,...,n$ and hence for all $A\text{\ensuremath{\in}}\mathcal{B}\left(\mathcal{C}\left(\mathbb{X};\mathbb{R}^{d}\right)\right)$
by Lemma \ref{lem:10} and Dynkin's $\pi-\lambda$ Theorem. 
\end{proof}
\begin{lem}
\label{lem:12}For all $n\geq1$, the mapping
\begin{align*}
\tilde{\kappa}_{n}:\mathbb{X}^{n}\times\mathbb{R}^{d\times n}\times\Theta & \rightarrow\Theta,\:\\
\left({\bf x}_{n},{\bf z}_{n},\left(m,k\right)\right) & \mapsto\left(m_{n}\left(\cdot;{\bf x}_{n},{\bf z}_{n}\right),k_{n}\left(\cdot;{\bf x}_{n}\right)\right)
\end{align*}
 is $\left(\mathcal{B}\left(\mathbb{X}^{n}\right)\otimes\mathcal{B}\left(\mathbb{R}^{d\times n}\right)\otimes\Sigma_{\Theta}\right)/\Sigma_{\Theta}$-measurable,
where ${\bf x}_{n}=\left(x_{1},...,x_{n}\right)\in\mathbb{X}^{n}$
and ${\bf z}_{n}=\left(z_{1},...,z_{n}\right)\in\mathbb{R}^{d\times n}$. 
\end{lem}
\begin{proof}
$\xi-m_{n}\left(\cdot;{\bf x}_{n},{\bf z}_{n}\right)$ is again a
multivariate Gaussian process with continuous sample paths and has
the covariance function $k_{n}\left(\cdot;{\bf x}_{n}\right)$ for
any deterministic design ${\bf x}_{n}=\left(x_{1},...,x_{n}\right)\in\mathbb{X}^{n}$
which implies indeed $\left(m_{n}\left(\cdot;{\bf x}_{n},{\bf z}_{n}\right),k_{n}\left(\cdot;{\bf x}_{n}\right)\right)\in\Theta$.
The claim for the measurability of the function $\tilde{\kappa}_{n}$
follows from the continuity of the mappings $\left(m,x\right)\mapsto m\left(x\right)$,
$\left(k,x\right)\mapsto k\left(x,\cdot\right)$,$\left(k,x,y\right)\mapsto k\left(x,y\right)$
and the measurability of the mapping $X\mapsto X^{\dagger}$ in combination
with the explicit expressions for $m_{n}\left(\cdot;{\bf x}_{n},{\bf z}_{n}\right)$
and $k_{n}\left(\cdot;{\bf x}_{n}\right)$ from Theorem \ref{thm:3}.
\end{proof}
We are now ready to prove the statement from Proposition \ref{prop:3}.
\begin{proof}
\textbf{(Proposition \ref{prop:3})}\\
Any $\nu\in\mathbb{M}$ is the distribution of a multivariate Gaussian
process $\xi$ and uniquely determined by its mean $m$ and covariance
$k$, so we can write $\nu=P^{\xi}=\mathcal{GP}_{d}\left(m,k\right)$
and define $\kappa_{n}:\mathbb{X}^{n}\times\mathbb{R}^{d\times n}\times\mathbb{M}\rightarrow\mathbb{M}$
by 
\[
\kappa_{n}\left({\bf x}_{n},{\bf z}_{n},\nu\right)=\mathcal{GP}_{d}\left(m_{n}\left(\cdot;{\bf x}_{n},{\bf z}_{n}\right),k_{n}\left(\cdot;{\bf x}_{n}\right)\right)=:P_{n}^{\xi}
\]
with ${\bf x}_{n}=\left(x_{1},...,x_{n}\right)\in\mathbb{X}^{n}$
and $\bm{z}_{n}=\left(z_{1,}...,z_{n}\right)\in\mathbb{R}^{d\times n}$.
The measurability of $\kappa_{n}$ follows by the previous two Lemmas
and the equality 
\[
\kappa_{n}\left({\bf x}_{n},{\bf z}_{n},\nu\right)=\Psi^{-1}\left(\tilde{\kappa}_{n}\left({\bf x}_{n},{\bf z}_{n},\Psi\left(\nu\right)\right)\right).
\]
We need to show that $P_{n}^{\xi}$ is a conditional distribution
of $\xi$ given the $\sigma$-algebra $\mathcal{F}_{n}$ generated
by the sequential design $\left(X_{n}\right)_{n\geq1}$ and pointwise
observations $\left(Z_{n}\right)_{n\geq1}$. By the defining property
of the conditional expectation this holds if we can prove 
\[
\mathbb{E}\left[UP_{n}^{\xi}\left(\Gamma\right)\right]=\mathbb{E}\left[U\boldsymbol{1}_{\xi\in\Gamma}\right]
\]
for any $U=\prod_{i=1}^{n}\varphi_{i}\left(Z_{i}\right)$
with measurable $\varphi_{i}:\b R^{d}\rightarrow\b R$ and $\Gamma\in\mathcal{B}\left(\mathcal{C}\left(\mathbb{X};\mathbb{R}^{d}\right)\right)$
of the form $\Gamma=\bigcap_{j=1}^{J}\left\{ \xi\left(\bar{x_{j}}\right)\in\Gamma_{j}\right\} $
with $\bar{x_{j}}\in\mathbb{X}$ and $\Gamma_{j}\in\mathcal{B}\left(\mathbb{R}^{d}\right)$
for $j=1,...,J$, since it extends to any $\c F_{n}$-measurable $U$
(recall that $X_{n}$ is $\c F_{n-1}$-measurable and by iteration
it can be written as a measurable function of $Z_{n-1},...,Z_{1}$)
and any set in $\mathcal{B}\left(\mathcal{C}\left(\mathbb{X};\mathbb{R}^{d}\right)\right)$
by Lemma \ref{lem:10} and Dynkin's $\pi-\lambda$ Theorem. Indeed,
the above statement follows by applying the equality 
\[
\kappa_{n+m}\left({\bf x}_{n+m},{\bf z}_{n+m},\nu\right)=\kappa_{m}\left({\bf x}_{n+1:n+m},{\bf z}_{n+1:n+m},\kappa_{n}\left({\bf x}_{n},{\bf z}_{n},\nu\right)\right)
\]
recursively to $P_{n}^{\xi}=\kappa_{n}\left({\bf x}_{n},{\bf z}_{n},P^{\xi}\right)$.
\end{proof}
\begin{proof}
\textbf{(Proposition \ref{prop:4})}\\
By Proposition \ref{prop:3} we have that the conditional distribution
of $\xi$ given $\mathcal{F}_{n}$ is of the form $P_{n}^{\xi}=\mathcal{GP}_{d}\left(m_{n},k_{n}\right)$
and that $\xi$ is a Bochner-integrable random element with values
in $\mathbb{S}$. We can define a Lévy-martingale by $\left(\mathbb{E}\left[\xi|\mathcal{F}_{n}\right]\right)_{n\in\mathbb{N}}$
with respect to the filtration $\left(\mathcal{F}_{n}\right)_{n\in\mathbb{N}}$,
which is again a random element with values in $\mathbb{S}$ for every
$n\in\b N$, and by the Convergence Theorem for Lévy-martingales (see
Theorem 6.1.12 in \cite{key-2}) we have 
\[
\mathbb{E}\left[\xi|\mathcal{F}_{n}\right]\xrightarrow[n\rightarrow\infty]{}\mathbb{E}\left[\xi|\mathcal{F}_{\infty}\right]
\]
in $\mathbb{S}$ with respect to the supremum norm, $P$-almost surely
and in $L^{1}\left(\Omega,\mathcal{F},P\right)$. The limit $m_{\infty}:=\mathbb{E}\left[\xi|\mathcal{F}_{\infty}\right]$
is again a random element with values in $\mathbb{S}$ and $\mathcal{F}_{\infty}$-measurable
by definition. We clearly have $m_{n}=\mathbb{E}\left[\xi|\mathcal{F}_{n}\right]$
since both are elements in $\mathbb{S}$ with $m_{n}\left(x\right)=\mathbb{E}\left[\xi\left(x\right)|\mathcal{F}_{n}\right]=\delta_{x}\left(\mathbb{E}\left[\xi|\mathcal{F}_{n}\right]\right)$
for all $x\in\mathbb{X}$, where $\delta_{x}$ denotes the (linear
and continuous) evaluation function $\delta_{x}:\b S\rightarrow\b R^{d}$
with $\delta_{x}\left(f\right):=f\left(x\right)$. We conclude 
\[
m_{n}\xrightarrow[n\rightarrow\infty]{a.s.}m_{\infty}
\]
 uniformly on $\mathbb{X}$.

Assume now for simplicity that $\xi$ is a centered Gaussian random
element so the covariance function reduces to $k^{c}\left(x,y\right)=\mathbb{E}\left[\xi\left(x\right)\xi\left(y\right)^{\top}\right]$.
We can define a random element $\xi^{2}$ in the separable Banach
space $\mathcal{C}\left(\mathbb{X}\times\mathbb{X};\mathbb{R}^{d\times d}\right)$
with supremum norm 
\[
\left\Vert f\right\Vert _{\infty}:=\sup_{x\in\b X}\left(\max_{i,j\text{\ensuremath{\in\left\{ 1,...,d\right\} }}}\left|f\left(x\right)_{ij}\right|\right)
\]
by $\xi^{2}\left(x,y\right):=\xi\left(x\right)\xi\left(y\right)^{\top}$
for $x,y\in\b X$. Note that it holds
\begin{align*}
\sup_{x,y\in\mathbb{X}}\left\Vert \xi\left(x\right)\xi\left(y\right)^{\top}\right\Vert _{F} & =\sup_{x,y\in\mathbb{X}}\text{tr}\left(\left(\xi\left(x\right)\xi\left(y\right)^{\top}\right)^{\top}\left(\xi\left(x\right)\xi\left(y\right)^{\top}\right)\right)^{\frac{1}{2}}\\
 & =\sup_{x,y\in\mathbb{X}}\left(\xi\left(x\right)^{\top}\xi\left(x\right)\right)^{\frac{1}{2}}\text{tr}\left(\xi\left(y\right)\xi\left(y\right)^{\top}\right)^{\frac{1}{2}}\\
 & =\sup_{x,y\in\mathbb{X}}\left\Vert \xi\left(x\right)\right\Vert _{2}\left\Vert \xi\left(y\right)\right\Vert _{2}\\
 & =\left(\sup_{x\in\mathbb{X}}\left\Vert \xi\left(x\right)\right\Vert _{2}\right)^{2}.
\end{align*}
Using the equivalence of norms on finite-dimensional vector spaces
and Fernique's Theorem, this yields $\b E\left[\left\Vert \xi^{2}\right\Vert _{\infty}\right]<\infty$
and hence $\xi^{2}$ is Bochner-integrable. By the same reasoning
as above we have $\mathbb{E}\left[\xi^{2}|\mathcal{F}_{n}\right]\xrightarrow[n\rightarrow\infty]{}\mathbb{E}\left[\xi^{2}|\mathcal{F}_{\infty}\right]=:k_{\infty}^{c}$
in $\mathcal{C}\left(\mathbb{X}\times\mathbb{X};\mathbb{R}^{d\times d}\right)$
with respect to the supremum norm, $P$-almost surely and in $L^{1}\left(\Omega,\mathcal{F},P\right)$.
Hence $k_{n}^{c}=\mathbb{E}\left[\xi^{2}|\mathcal{F}_{n}\right]$,
since both are elements in $\mathcal{C}\left(\mathbb{X}\times\mathbb{X};\mathbb{R}^{d\times d}\right)$
with $k_{n}^{c}\left(x,y\right)=\mathbb{E}\left[\xi\left(x\right)\xi\left(y\right)^{\top}|\mathcal{F}_{n}\right]=\delta_{x,y}\left(\mathbb{E}\left[\xi^{2}|\mathcal{F}_{n}\right]\right)$
for all $\left(x,y\right)\in\mathbb{X}\times\mathbb{X}$ and hence
\[
k_{n}^{c}\xrightarrow[n\rightarrow\infty]{a.s.}k_{\infty}^{c}
\]
 uniformly on $\mathbb{X}\times\mathbb{X}$. The general case (non-centered)
follows in combination with the first part since 
\begin{align*}
k_{n}\left(x,y\right) & =\mathbb{E}\left[\left(\xi\left(x\right)-m\left(x\right)\right)\left(\xi\left(y\right)-m\left(y\right)\right)^{\top}|\mathcal{F}_{n}\right]\\
 & =k_{n}^{c}\left(x,y\right)-m_{n}\left(x\right)m_{n}\left(y\right)^{\top}.
\end{align*}
Hence we conclude 
\[
k_{n}\xrightarrow[n\rightarrow\infty]{a.s.}k_{\infty}
\]
 uniformly on $\mathbb{X}\times\mathbb{X}$.

Let now $Q$ denote any conditional distribution of $\xi$ given $\mathcal{F}_{\infty}$.
The $\mathcal{F}_{\infty}$-measurable random measure $Q$ is then
almost surely Gaussian (follows as in the proof of Proposition 2.9
in \cite{key-1} using the characteristic function for random vectors)
and thus taking 
\[
P_{\infty}^{\xi}\left(\omega,\cdot\right)=\begin{cases}
Q\left(\omega,\cdot\right) & ,\omega\in\Omega_{0}\\
\mathcal{GP}_{d}\left(0,0\right) & ,\omega\in\Omega\setminus\Omega_{0}
\end{cases}
\]
we have constructed a $\mathcal{F}_{\infty}$-measurable random element
in $\mathbb{M}$ such that 
\[
P_{n}^{\xi}\xrightarrow[n\rightarrow\infty]{a.s.}P_{\infty}^{\xi}.
\]
\end{proof}
\begin{proof}
\textbf{(Proposition \ref{prop:5})}\\
Let $\nu=\mathcal{GP}_{d}\left(m_{\nu},k_{\nu}\right)\in\mathbb{M}$
and let $\left(x_{i},z_{i}\right)\rightarrow\left(x,z\right)$ in
$\mathbb{X}\times\mathbb{R}^{d}$. For any $i\in\mathbb{N}\cup\left\{ \infty\right\} $
we have $\text{Cond}{}_{x_{i},z_{i}}\left(\nu\right)=\mathcal{GP}_{d}\left(m_{1}\left(\cdot;x_{i},z_{i}\right),k_{1}\left(\cdot,\cdot;x_{i}\right)\right)$
where $m_{1}$ and $k_{1}$ are given as in Theorem \ref{thm:3}.
It follows easily by uniform continuity of $m_{\nu}$ (see Lemma \ref{lem:2})
that $m_{1}\left(\cdot;x_{i},z_{i}\right)\rightarrow m_{1}\left(\cdot;x,z\right)$
uniformly on $\mathbb{X}$ and by uniform continuity of $k_{\nu}$
(Lemma \ref{lem:2}, $\mathbb{X}\times\mathbb{X}$ compact) together
with the continuity of $M\mapsto M^{\dagger}$ for matrices with the
same rank that $k_{1}\left(\cdot,\cdot;x_{i}\right)\rightarrow k_{1}\left(\cdot,\cdot;x\right)$
uniformly on $\mathbb{X}\times\mathbb{X}$.
\end{proof}

\subsection{Proofs of Section 4\label{subsec:app:sec4}}

\begin{proof}
\textbf{(Lemma \ref{lem:3})} \\
1.) Without loss of generality assume $\mathcal{H}_{0}=0$, since
it only adds a constant term. Furthermore $J_{n}\left(x\right)=\mathcal{J}_{x}\left(P_{n}^{\xi}\right)$
and hence it is equivalent to prove that the result holds for all
$P^{\xi}\in\mathbb{M}$ at $n=0$.

Assume now that $n=0$ , $x\in\mathbb{X}$ such that 
\[
\Sigma\left(x\right)=\Sigma_{0}\left(x\right)=k\left(x,x\right)+\mathcal{T}\left(x\right)\in\mathbb{R}^{d\times d}
\]
has rank $k$ for some $k\in\left\{ 0,...,d\right\} $ and let $\left(x_{i}\right)_{i\in\mathbb{N}}$
be a sequence in $C_{0,k}$ with $x_{i}\xrightarrow{i\rightarrow\infty}x$
($C_{0,k}$ is separable as subset of a separable metric space for
all $k\in\left\{ 0,...,d\right\} $). Recall that it holds
\[
J_{0}\left(x\right)=\mathcal{J}_{x}\left(P^{\xi}\right)=\mathbb{E}\left[\mathcal{H}\left(\text{Cond}{}_{x,Z(x)}\left(P^{\xi}\right)\right)\right],
\]
so if we take $\bm{\nu}_{i}:=\text{Cond}{}_{x_{i},Z(x_{i})}\left(P^{\xi}\right)$
for $i\in\mathbb{N}$ and $\bm{\nu}_{\infty}:=\text{Cond}{}_{x,Z(x)}\left(P^{\xi}\right)$
we have $\bm{\nu}_{i}\in\mathfrak{P}\left(\xi\right)$ and by Proposition
\ref{prop:5} $\bm{\nu}_{i}\xrightarrow[i\rightarrow\infty]{a.s.}\bm{\nu}_{\infty}$.
It follows 
\[
\mathcal{H}\left(\bm{\nu}_{i}\right)\xrightarrow[i\rightarrow\infty]{a.s.}\mathcal{H}\left(\bm{\nu}_{\infty}\right)
\]
by $\mathfrak{P}-$continuity of $\mathcal{H}$ and by the above equality
finally 
\[
\mathcal{J}_{x_{i}}\left(P^{\xi}\right)\xrightarrow[i\rightarrow\infty]{}\mathcal{J}_{x}\left(P^{\xi}\right)
\]
since $\left(\mathcal{H}\left(\bm{\nu}_{i}\right)\right)_{i\in\mathbb{N}}$
is uniformly integrable.

2.) Let $n\in\b N$. Assume that we have a deterministic design such
that $X_{i}=x_{i}$ and $Z_{i}=Z_{i}\left(x_{i}\right)$ for all $i\leq n$.
We will first prove that $k_{n}\left(x,x\right)$ is positive definite
for all $n\in\b N$ and $x\in\b X\setminus\left\{ x_{1},...,x_{n}\right\} $.
As already mentioned in the proof of Theorem \ref{thm:3} we know
that $\left(\xi\left(x\right)^{\top},\boldsymbol{Z}_{n}^{\top}\right)^{\top}$,
where $\boldsymbol{Z}_{n}:=\left(Z_{1}^{\top},...,Z_{n}^{\top}\right)^{\top}$,
is a Gaussian vector by definition of multivariate Gaussian processes
with covariance matrix 
\[
\left(\begin{array}{cc}
\text{Var}\left(\xi\left(x\right)\right) & \text{Cov}\left(\xi\left(x\right),\boldsymbol{Z}_{n}\right)\\
\text{Cov}\left(\boldsymbol{Z}_{n},\xi\left(x\right)\right) & \text{Var}\left(\boldsymbol{Z}_{n}\right)
\end{array}\right)=\left(\begin{array}{cc}
k\left(x,x\right) & K(x,\bm{x}_{n})\\
K(\bm{x}_{n},x) & \Sigma\left(\boldsymbol{x}_{n}\right)
\end{array}\right),
\]
using the same notation as in Theorem \ref{thm:3}. The covariance
matrix above is positive definite whenever $x\in\b X\setminus\left\{ x_{1},...,x_{n}\right\} $,
since for every $v_{0},v_{1},...,v_{n}\in\b R^{d}$, at least one
non-zero, it holds 
\[
\sum_{i,j=0}^{n}v_{i}^{\top}k\left(x_{i},x_{j}\right)v_{j}+\sum_{i=1}^{n}v_{i}^{\top}\c T\left(x_{i}\right)v_{i}>0
\]
 by positive definiteness of $k$, where we define $x_{0}:=x$. As
shown in Theorem \ref{thm:3} the covariance of $\text{\ensuremath{\xi\left(x\right)}}$
given $\boldsymbol{Z}_{n}$ is 
\[
k_{n}\left(x,x\right)=k\left(x,x\right)-K(x,\bm{x}_{n})\Sigma(\bm{x}_{n})^{-1}K(x,\bm{x}_{n})^{\top},
\]
which is exactly the Schur complement of the covariance matrix stated
above. By \cite{key-29} the Schur complement of a positive definite
matrix is again positive definite and hence we conclude that $k_{n}\left(x,x\right)$
is positive definite for all $x\in\b X\setminus\left\{ x_{1},...,x_{n}\right\} $.

For $x\in\left\{ x_{1},...,x_{n}\right\} $ we obtain again a Gaussian
vector $\left(\xi\left(x_{i}\right),\boldsymbol{Z}_{n}\right)$ with
covariance matrix 
\begin{align*}
 & \left(\begin{array}{cc}
\text{Var}\left(\xi\left(x_{i}\right)\right) & \text{Cov}\left(\xi\left(x_{i}\right),\boldsymbol{Z}_{n}\right)\\
\text{Cov}\left(\boldsymbol{Z}_{n},\xi\left(x_{i}\right)\right) & \text{Var}\left(\boldsymbol{Z}_{n}\right)
\end{array}\right)\\
= & \left(\begin{array}{cc}
k\left(x_{i},x_{i}\right) & K(x_{i},\bm{x}_{n})\\
K(\bm{x}_{n},x_{i}) & K\left(\boldsymbol{x}_{n}\right)
\end{array}\right)+\left(\begin{array}{cc}
0 & 0\\
0 & \c T\left(\boldsymbol{x}_{n}\right)
\end{array}\right).
\end{align*}
The first matrix is positive semi-definite as covariance matrix of
a Gaussian vector (one entry is double) and the second matrix contains
the positive definite matrix $\c T\left(\boldsymbol{x}_{n}\right)$.
Since $k\left(x_{i},x_{i}\right)$ is also positive definite, we conclude
\[
\left(v_{1},v_{2}\right)^{\top}\left(\begin{array}{cc}
\text{Var}\left(\xi\left(x_{i}\right)\right) & \text{Cov}\left(\xi\left(x_{i}\right),\boldsymbol{Z}_{n}\right)\\
\text{Cov}\left(\boldsymbol{Z}_{n},\xi\left(x_{i}\right)\right) & \text{Var}\left(\boldsymbol{Z}_{n}\right)
\end{array}\right)\left(v_{1},v_{2}\right)>0
\]
for all $\left(v_{1},v_{2}\right)\in\mathbb{R}^{d}\times\b R^{nd}$,
$\left(v_{1},v_{2}\right)\neq0$ and hence the positive definiteness
of the covariance matrix of $\left(\xi\left(x_{i}\right),\boldsymbol{Z}_{n}\right)$.
This means also the Schur complement $k_{n}\left(x_{i},x_{i}\right)$
is positive definite. We finally conclude that $k_{n}\left(x,x\right)$
is positive definite for all $x\in\b X$. 

If $k_{n}\left(x,x\right)$ is positive definite, then $\Sigma_{n}\left(x\right)$
has full rank since $\text{rank}\left(\Sigma_{n}\left(x\right)\right)\geq\max\left\{ \text{rank}\left(k_{n}\left(x,x\right)\right),\text{rank}\left(\c T\left(x\right)\right)\right\} =d$
and hence the sample paths of $J_{n}$ are continuous on $\b X$ by
1.). 

Since $J_{n}$ has continuous sample paths and $\b X$ is compact,
we know that $A_{n}\left(\omega\right):=\left\{ x\in\b X:J_{n}\left(x\right)\left(\omega\right)=\inf_{x\in\b X}J_{n}\left(x\right)\left(\omega\right)\right\} $
is a non-empty closed set for every $\omega\in\Omega$. Hence the
mapping $\omega\mapsto A_{n}\left(\omega\right)$ is a $\mathcal{F}_{n}$-measurable
random closed set that admits a $\mathcal{F}_{n}$-measurable selection
$X_{n+1}$, i.e. a $\b X$-valued random element such that $X_{n+1}\left(\omega\right)\in A_{n}\left(\omega\right)$
for all $\omega\in\Omega$ (see Theorem 2.13 in \cite{key-28}). 

3.) Define for $\omega\in\Omega$ and $0\leq k\leq d$ the sets 
\[
C_{n,k}^{\leq}\left(\text{\ensuremath{\omega}}\right):=\left\{ x\in\mathbb{X}:\text{rank}\left(\Sigma_{n}\left(x\right)\right)\left(\omega\right)\leq k\right\} ,
\]

\[
M_{n}\left(\omega\right):=\left\{ x\in\mathbb{X}:J_{n}\left(x\right)\left(\omega\right)\leq\inf_{y\in\mathbb{X}}J_{n}\left(y\right)\left(\omega\right)+\epsilon_{n}\right\} .
\]
For each $0\leq k\leq d$ the set $C_{n,k}^{\leq}\left(\text{\ensuremath{\omega}}\right)$
is compact as closed subsets of the compact metric space $\mathbb{X}$
by continuity of the mapping $x\mapsto\left(\Sigma_{n}\left(x\right)\right)\left(\omega\right)$
and closeness of the set $\left\{ M\in\mathbb{R}^{d\times d}:\text{rank}(M)\leq k\right\} $.
Since $J_{n}\left(\cdot\right)\left(\omega\right)$ is by definition
constant if $\Sigma_{n}\left(\cdot\right)\left(\omega\right)\equiv0$
and hence $M_{n}\left(\omega\right)=\mathbb{X}$, we assume that $\omega\in\Omega$
is such that $\Sigma_{n}\left(\cdot\right)\left(\omega\right)\not\equiv0$.
Without loss of generality we assume $\inf_{y\in\mathbb{X}}J_{n}\left(y\right)\left(\omega\right)<H_{n}\left(\omega\right)$
and check the following two cases.

Assume $C_{n,0}^{\leq}\left(\text{\ensuremath{\omega}}\right)\cap M_{n}\left(\omega\right)\neq\emptyset$
. Then it holds for $x_{*}\in C_{n,0}^{\leq}\left(\text{\ensuremath{\omega}}\right)\cap M_{n}\left(\omega\right)$
that $\Sigma_{n}\left(x_{*}\right)\left(\omega\right)=0$ and hence
$J_{n}\left(x_{*}\right)\left(\omega\right)=H_{n}\left(\omega\right)$,
which yields
\[
\inf_{y\in\mathbb{X}}J_{n}\left(y\right)\left(\omega\right)\leq J_{n}\left(x\right)\left(\omega\right)\leq\inf_{y\in\mathbb{X}}J_{n}\left(y\right)\left(\omega\right)+\epsilon_{n}
\]
for all $x\in\mathbb{X}$ and hence $M_{n}\left(\omega\right)=\mathbb{X}$.

Assume $C_{n,0}^{\leq}\left(\text{\ensuremath{\omega}}\right)\cap M_{n}\left(\omega\right)=\emptyset$.
Then there must exist an integer $1\leq k_{*}\leq d$ such that $C_{n,k}^{\leq}\left(\text{\ensuremath{\omega}}\right)\cap M_{n}\left(\omega\right)=\emptyset$
for all $0\leq k<k_{*}$ and $C_{n,k_{*}}^{\leq}\left(\text{\ensuremath{\omega}}\right)\cap M_{n}\left(\omega\right)\neq\emptyset$.
Indeed, if $\left(x_{j}\right)_{j\in\mathbb{N}}\subset\mathbb{X}$
is a sequence such that $J_{n}\left(x_{j}\right)\left(\omega\right)\rightarrow\inf_{y\in\mathbb{X}}J_{n}\left(y\right)\left(\omega\right)$
for $j\rightarrow\infty$, then we have for some $j$ large enough
that 
\[
J_{n}\left(x_{j}\right)\left(\omega\right)\leq\inf_{y\in\mathbb{X}}J_{n}\left(y\right)\left(\omega\right)+\epsilon_{n}
\]
and $J_{n}\left(x_{j}\right)\left(\omega\right)<H_{n}\left(\omega\right)$,
which implies $\text{tr}\left(\Sigma_{n}\left(x_{j}\right)\right)\left(\omega\right)>0$
and hence $x_{j}\in C_{n,k}^{\leq}\left(\text{\ensuremath{\omega}}\right)\cap M_{n}\left(\omega\right)$
for some $1\leq k\leq d$. Hence the sets $C_{n,k}^{\leq}\left(\text{\ensuremath{\omega}}\right)\cap M_{n}\left(\omega\right)$
are not empty for some $k$ large enough. Taking the smallest integer
$k_{*}\geq1$ such that $C_{n,k_{*}}^{\leq}\left(\text{\ensuremath{\omega}}\right)\cap M_{n}\left(\omega\right)\neq\emptyset$
gives the desired result. 

Choosing $k_{*}$ this way means
\[
C_{n,k_{*}}^{\leq}\left(\text{\ensuremath{\omega}}\right)\cap M_{n}\left(\omega\right)\subset C_{n,k_{*}}
\]
and since $J_{n}$ is continuous on $C_{n,k_{*}}$ by 1.), we conclude
that $C_{n,k_{*}}^{\leq}\left(\text{\ensuremath{\omega}}\right)\cap M_{n}\left(\omega\right)$
is closed and hence compact. This means the mapping $\omega\mapsto C_{n,k_{*}}^{\leq}\left(\text{\ensuremath{\omega}}\right)\cap M_{n}\left(\omega\right)$
is an $\mathcal{F}_{n}$-measurable random closed set and we can chose
an $\mathcal{F}_{n}$-measurable selection $X_{n+1}^{(k_{*})}$ that
takes values in this random closed set (see Theorem 2.13 in \cite{key-28}).
The desired $\epsilon$-quasi SUR sequential design $\left(X_{n}\right)_{n\geq1}$
can hence be defined by 
\[
X_{n+1}=\begin{cases}
x & ,M_{n}=\mathbb{X}\\
X_{n+1}^{(k_{*})} & ,\text{else}
\end{cases}
\]
for some arbitrary $x\in\mathbb{X}$.
\end{proof}
\begin{proof}
\textbf{(Proposition \ref{prop:6})} \\
Using the notation from the proof of Proposition \ref{prop:3}, we
see that 
\[
\mathcal{J}_{x}\left(\nu\right)=\int_{\mathbb{R}^{d}}\mathcal{H}\left(\kappa_{1}\left(x,m_{\nu}\left(x\right)+\Sigma_{\nu}\left(x\right)^{\frac{1}{2}}u,\nu\right)\right)\phi_{d}\left(u\right)du.
\]
Using Lemma \ref{lem:12} in the Appendix and the measurability of
$\kappa_{1}$ (see proof of Proposition \ref{prop:3}), we see that
the integrand is a $\mathcal{B}\left(\mathbb{X}\right)\otimes\mathcal{B}\left(\mathbb{R}^{d}\right)\otimes\mathcal{M}$
measurable function of $\left(x,u,\nu\right)$. The claim follows
by Fubini's Theorem.
\end{proof}

\subsection{Proofs of Section 5\label{subsec:app:sec5}}

\begin{thm}
\label{thm:7}Let $\mathcal{H}$ be an uncertainty functional on $\mathbb{M}$
that has the supermartingale property, $\mathcal{G}$ the associated
maximal expected gain functional and $\left(X_{n}\right)_{n\geq1}$
be an $\epsilon$-quasi SUR sequential design for $\mathcal{H}$.
\begin{enumerate}
\item Then it holds
\[
\mathcal{G}\left(P_{n}^{\xi}\right)\xrightarrow[n\rightarrow\infty]{a.s.}0.
\]
\item If moreover 
\begin{enumerate}
\item $H_{n}:=\mathcal{H}\left(P_{n}^{\xi}\right)\xrightarrow[n\rightarrow\infty]{a.s.}\mathcal{H}\left(P_{\infty}^{\xi}\right),$
\item $\mathcal{G}\left(P_{n}^{\xi}\right)\xrightarrow[n\rightarrow\infty]{a.s.}\mathcal{G}\left(P_{\infty}^{\xi}\right),$
\item $\mathbb{Z}_{\mathcal{H}}:=\left\{ \nu\in\mathbb{M}:\mathcal{H}\left(\nu\right)=0\right\} =\left\{ \nu\in\mathbb{M}:\mathcal{G}\left(\nu\right)=0\right\} =:\mathbb{Z}_{\mathcal{G}},$ 
\end{enumerate}
then we have 
\[
H_{n}:=\mathcal{H}\left(P_{n}^{\xi}\right)\xrightarrow[n\rightarrow\infty]{a.s.}0.
\]

\end{enumerate}
\end{thm}
\begin{proof}
See Theorem 3.12 in \cite{key-1}. The proof works the same way for
multivariate Gaussian processes.
\end{proof}
\begin{thm}
\label{thm:8}Let $\mathcal{H}$ be an uncertainty functional on $\mathbb{M}$
and $\mathcal{G}$ the associated maximal expected gain functional.
Assume that we can decompose $\mathcal{H}=\mathcal{H}_{0}+\mathcal{H}_{1}$,
where
\begin{enumerate}
\item $\mathcal{H}_{0}\left(\nu\right)=\int_{\mathcal{\mathbb{S}}}L_{0}d\nu$
for some $L_{0}\in\bigcap_{\nu\in\mathbb{M}}\mathcal{L}^{1}\left(\mathbb{S},\mathcal{S},\nu\right)$
and 
\item $\mathcal{H}_{1}$ is $\mathfrak{P}$-continuous, $\mathfrak{P}$-uniformly
integrable and has the supermartingale property.
\end{enumerate}
Then, for any quasi-SUR sequential design associated with $\mathcal{H}$,
it holds 
\[
\mathcal{G}\left(P_{\infty}^{\xi}\right)\overset{a.s.}{=}0.
\]
\end{thm}
\begin{proof}
Using Theorem \ref{thm:7} and Proposition \ref{prop:4}, it is straightforward
to obtain 
\[
\mathcal{G}_{x}\left(P_{\infty}^{\xi}\right)=\mathcal{H}\left(P_{\infty}^{\xi}\right)-\mathbb{E}_{\infty}\left[\mathcal{H}\left(P_{\infty,x}^{\xi}\right)\right]\stackrel{a.s.}{=}0
\]
for each $x\in\mathbb{X}$ by following the same steps as in the proof
of Theorem 3.16 in \cite{key-1}. To conclude we need to show that
\[
\mathcal{G}\left(P_{\infty}^{\xi}\right)=\sup_{x\in\mathbb{X}}\mathcal{G}_{x}\left(P_{\infty}^{\xi}\right)\stackrel{a.s.}{=}0.
\]
Define with $\Sigma_{\infty}\left(x\right)=k_{\infty}\left(x,x\right)+\mathcal{T}\left(x\right)$
the random sets 
\[
C_{k,\infty}:=\left\{ x\in\mathbb{X}:\text{rank}\left(\Sigma_{\infty}\left(x\right)\right)=k\right\} 
\]
and 
\[
C_{k,\infty}^{\leq}:=\left\{ x\in\mathbb{X}:\text{rank}\left(\Sigma_{\infty}\left(x\right)\right)\leq k\right\} 
\]
for $k=0,...,d$. For $k\in\left\{ 0,...,d\right\} $ the sets $C_{k,\infty}^{\leq}$
are closed subsets of $\b X$ ($\Sigma_{\infty}\left(\omega\right)$
is continuous for all $\omega\in\Omega$ and $\left\{ M\in\Rdd:\text{rank}\left(M\right)\leq k\right\} $
is closed) and hence compact and separable with $C_{k,\infty}^{\leq}\subseteq C_{k+1,\infty}^{\leq}$,
$k\in\left\{ 0,...,d-1\right\} $, and $\b X=C_{d,\infty}^{\leq}$.
By the previous Lemma we know that the sample paths of $J_{\infty}$
are continuous on each set $C_{k,\infty}$ and hence $x\mapsto\mathcal{G}_{x}\left(P_{\infty}^{\xi}\right)$
has continuous sample paths on $C_{k,\infty}$. The statement follows by iteration.

Let $\left\{ x_{i}\right\} _{i\in\b N}$ be a countable dense subset
of $C_{1,\infty}^{\leq}$. In the previous part we have seen that
$\mathcal{G}_{x_{i}}\left(P_{\infty}^{\xi}\right)=0$ for all $i\in\b N$,
almost surely. Using the continuity of $x\mapsto\mathcal{G}_{x}\left(P_{\infty}^{\xi}\right)$
on $C_{1,\infty}$ and the fact that $\mathcal{G}_{x}\left(P_{\infty}^{\xi}\right)=0$
on $C_{0,\infty}$, we conclude 
\[
\sup_{x\in C_{1,\infty}^{\leq}}\mathcal{G}_{x}\left(P_{\infty}^{\xi}\right)=0
\]
almost surely. Assume now that we have shown $\mathcal{G}_{x}\left(P_{\infty}^{\xi}\right)=0$
on $C_{k,\infty}^{\leq}$ almost surely. Let $\left\{ x_{i}\right\} _{i\in\b N}$
be a countable dense subset of $C_{k+1,\infty}^{\leq}$. Then $\mathcal{G}_{x_{i}}\left(P_{\infty}^{\xi}\right)=0$
for all $i\in\b N$ almost surely and using the continuity of $x\mapsto\mathcal{G}_{x}\left(P_{\infty}^{\xi}\right)$
on $C_{k+1,\infty}$ and $\mathcal{G}_{x}\left(P_{\infty}^{\xi}\right)=0$
on $C_{k,\infty}^{\leq}$ we conclude
\[
\sup_{x\in C_{k+1,\infty}^{\leq}}\mathcal{G}_{x}\left(P_{\infty}^{\xi}\right)=0
\]
almost surely. This leads in the end for $k=d$ to 
\[
\mathcal{G}\left(P_{\infty}^{\xi}\right)=\sup_{x\in\mathbb{X}}\mathcal{G}_{x}\left(P_{\infty}^{\xi}\right)\stackrel{}{=}0
\]
almost surely, since $\b X=C_{d,\infty}^{\leq}$.
\end{proof}
\begin{proof}
\textbf{(Proposition \ref{prop:7})}\\
Take $\mathcal{H}_{0}=0$ and $\mathcal{H}_{1}=\mathcal{H}$. Then
we have for any quasi-SUR sequential design associated with $\mathcal{H}$
that $\mathcal{G}\left(P_{\infty}^{\xi}\right)\overset{a.s.}{=}0$
by Theorem \ref{thm:8}. By the first part of Theorem \ref{thm:7}
it holds $\mathcal{G}\left(P_{n}^{\xi}\right)\xrightarrow[n\rightarrow\infty]{a.s.}0$
and hence $\mathcal{G}\left(P_{n}^{\xi}\right)\xrightarrow[n\rightarrow\infty]{a.s.}\mathcal{G}\left(P_{\infty}^{\xi}\right)$.
By $\mathfrak{P}$-continuity we also have $\mathcal{H}\left(P_{n}^{\xi}\right)\xrightarrow[n\rightarrow\infty]{a.s.}\mathcal{H}\left(P_{\infty}^{\xi}\right)$
and hence all conditions for the second part of Theorem \ref{thm:7}
are satisfied. We conclude $H_{n}\xrightarrow[n\rightarrow\infty]{a.s.}0$.
\end{proof}

\subsubsection{Proofs of Section 5.1 for IBV}

\begin{proof}
\textbf{(Lemma \ref{lem:4})} \\
$\mathcal{H}^{IBV}$ is $\mathfrak{P}$-uniformly integrable, since
the uncertainty functional is upper-bounded. Indeed, we have for every
Gaussian random element $\xi$ in $\b S$ and $u\in\mathbb{X}$
\[
\text{Var}\left({\bf 1}_{\Gamma\left(\xi\right)}\left(u\right)\right)\overset{}{\leq}\frac{1}{4},
\]
since ${\bf 1}_{\Gamma\left(\xi\right)}\left(u\right)\in\left[0,1\right]$.
Hence we have for any measure $\nu\in\mathbb{M}$ taking $\xi\sim\nu$
that
\[
\left|\mathcal{H}^{IBV}\left(\nu\right)\right|=\int_{\mathbb{X}}\text{Var}\left({\bf 1}_{\Gamma\left(\xi\right)}\left(u\right)\right)\mu\left(du\right)\leq\frac{\mu\left(\b X\right)}{4},
\]
since $\mu$ is a finite measure over $\mathbb{X}$. Furthermore, $H_{n}^{IBV}$ is an $\c F_{n}$-measurable random variable
and integrable by definition of the conditional expectation. Furthermore,
we have for $p_{n}\left(u\right)=\mathbb{E}\left[{\bf 1}_{\Gamma\left(\xi\right)}\left(u\right)|\mathcal{F}_{n}\right]$
by the tower property and Jensen's inequality 
\begin{align*}
\E\left[H_{n}^{IBV}|\mathcal{F}_{n-1}\right] & =\int_{\mathbb{X}}\left(\E\left[p_{n}\left(u\right)|\c F_{n-1}\right]-\E\left[p_{n}\left(u\right)^{2}|\c F_{n-1}\right]\right)\mu\left(du\right)\\
 & \overset{a.s.}{\leq}\int_{\mathbb{X}}\left(p_{n-1}\left(u\right)-p_{n-1}\left(u\right)^{2}\right)\mu\left(du\right)\\
 & =H_{n-1}^{IBV}.
\end{align*}
\end{proof}
\begin{proof}
\textbf{(Lemma \ref{lem:13}) }\\
That $\mathbb{X}$ can be written as the disjoint union 
\[
\bigcup_{J_1,J_2\subseteq\left\{ 1,...,d\right\} }B_{J_1,J_2}\left(\omega\right)
\]
follows already by assumption 3. of the Lemma. 

Let $\xi$ be a random element in $\b S$ with $\xi\sim\c{GP}_{d}\left(m,k\right)$
and let $\bm{\nu}_{n}\in\mathfrak{P}\left(\xi\right)$ for all $n\in\mathbb{N}\cup\left\{ \infty\right\} $
such that $\bm{\nu}_{n}\xrightarrow{}\bm{\nu}_{\infty}$
almost surely. By definition of $\mathfrak{P}\left(\xi\right)$ there
exist $\sigma$-algebras $\c G_{n}$ such that $\bm{\nu}_{n}=P\left(\xi\in\cdot|\mathcal{G}_{n}\right)$
for all $n\in\mathbb{N}\cup\left\{ \infty\right\} $. 
For the first property note that for $u\in\b X$ and $j\in\left\{ 1,...,d\right\} $
we have
\begin{align*}
 & \b E\left[\I_{k_{\infty}\left(u,u\right)_{jj}=0}\left(\xi\left(u\right)_{j}-m_{\infty}\left(u\right)_{j}\right)^{2}\right]\\
= & \b E\left[\I_{k_{\infty}\left(u,u\right)_{jj}=0}\b E\left[\left(\xi\left(u\right)_{j}-m_{\infty}\left(u\right)_{j}\right)^{2}|\c G_{\infty}\right]\right]\\
= & \b E\left[\I_{k_{\infty}\left(u,u\right)_{jj}=0}k_{\infty}\left(u,u\right)_{jj}\right]\\
= & 0
\end{align*}
by the tower property and since $k_{\infty}\left(u,u\right)$ is $\c G_{\infty}$-measurable.
This means 
\[
\I_{k_{\infty}\left(u,u\right)_{jj}=0}\left(\xi\left(u\right)_{j}-m_{\infty}\left(u\right)_{j}\right)^{2}\overset{a.s.}{=}0,
\]
since it is a non-negative random variable, and hence for almost all
$\omega\in\Omega$ we have that $\sum_{j\in J_2}k_{\infty}\left(u,u\right)\left(\text{\ensuremath{\omega}}\right)_{jj}^{2}=0$
implies 
\[
\xi\left(u\right)\left(\text{\ensuremath{\omega}}\right)_{j}=m_{\infty}\left(u\right)\left(\text{\ensuremath{\omega}}\right)_{j}
\]
for all $j\in J_2$. Note that $\left(u,\omega\right)\mapsto m_{\infty}\left(u\right)\left(\omega\right)$
and $\left(u,\omega\right)\mapsto k_{\infty}\left(u,u\right)\left(\omega\right)$
are jointly measurable by continuity of the sample paths $u\mapsto m_{\infty}\left(u\right)\left(\omega\right)$
and $u\mapsto k_{\infty}\left(u,u\right)\left(\omega\right)$ for
all $\omega\in\Omega$. By Fubini-Tonelli we conclude
\begin{align*}
 & \b E\left[\mu\left(B_{J_{1},J_{2}}\right)\right]\\
= & \int_{\b X}\left(\I_{\sum_{j\in J_{1}}k\left(u,u\right)_{jj}^{2}=0}\prod_{j\in J_{1}^{c}}\I_{k\left(u,u\right)_{jj}>0}\right) \cdot\\
 & \b E\left[\I_{\sum_{j\in J_{2}}k_{\infty}\left(u,u\right)_{jj}^{2}=0} \prod_{j\in J_{2}}\I_{m_{\infty}\left(u\right)_{j}=\t_{j}} \cdot \right.
 \\ & \left. \quad \prod_{j\in J_{2}^{c}}\max\left(\I_{k_{\infty}\left(u,u\right)_{jj}>0},\I_{m_{\infty}\left(u\right)_{j}\neq \t_{j}}\right)\right]\mu\left(du\right)\\
\leq & \int_{\b X}\left(\I_{\sum_{j\in J_{1}}k\left(u,u\right)_{jj}^{2}=0}\prod_{j\in J_{1}^{c}}\I_{k\left(u,u\right)_{jj}>0}\right)\b \cdot \\ & E\left[\I_{\sum_{j\in J_{2}}k_{\infty}\left(u,u\right)_{jj}^{2}=0}\prod_{j\in J_{2}}\I_{\xi\left(u\right)_{j}=\t_{j}}\right]\mu\left(du\right)\\
\leq & \int_{\b X}\left(\I_{\sum_{j\in J_{1}}k\left(u,u\right)_{jj}^{2}=0}\prod_{j\in J_{1}^{c}}\I_{k\left(u,u\right)_{jj}>0}\right) \b E\left[\prod_{j\in J_{2}}\I_{\xi\left(u\right)_{j}=\t_{j}}\right]\mu\left(du\right)\\
= & 0,
\end{align*}
where the last equality follows by the assumption that there exists
$j^{*}\in\left\{ 1,...,d\right\} $ with $j^{*}\in J_{2}$ and
$j^{*}\in J_{1}^{c}$. Indeed, we have for the multivariate Gaussian
process $\xi$ that $\xi\left(u\right)_{j^{*}}\sim\mathcal{N}\left(m\left(u\right)_{j^{*}},k\left(u,u\right)_{j^{*}j^{*}}\right)$
which implies $\xi\left(u\right)_{j^{*}}\overset{a.s.}{\neq}\t_{j^{*}}$
since $k\left(u,u\right)_{j^{*}j^{*}}>0$ and hence 
\begin{align*}
\b E\left[\prod_{j\in J_{2}}\I_{\xi\left(u\right)_{j}=\t_{j}}\right] & =P\left(\bigcap_{j\text{\ensuremath{\in J_{2}}}}\left\{ \xi\left(u\right)_{j}=\t_{j}\right\} \right)\\
 & \leq\min_{j\in J_{2}}P\left(\xi\left(u\right)_{j}=\t_{j}\right)\\
 & \leq P\left(\xi\left(u\right)_{j^{*}}=\t_{j^{*}}\right)\\
 & =0.
\end{align*}
Since $\mu\left(B_{J_{1},J_{2}}\right)$ is almost surely non-negative,
the first property follows.

We will now turn to the second property. For $j\in\left\{ 1,...,d\right\} $
we have that $k\left(u,u\right)_{jj}=0$ implies $k_{n}\left(u,u\right)_{jj}\overset{a.s.}{=}0$
for all $n\in\text{\ensuremath{\mathbb{N}} }\cup\left\{ \infty\right\} $.
Indeed, for a random variable $X$ and $\sigma$-algebra $\c F$ 
\[
\b E\left[\left(X-\b E\left[X\right]\right)^{2}\right]=0
\]
implies $X\overset{a.s.}{=}\b E\left[X\right]$ and hence also $\b E\left[X|\c F\right]\overset{a.s.}{=}X$.
We conclude 
\[
\b E\left[\left(X-\b E\left[X|\c F\right]\right)^{2}|\c F\right]\overset{a.s.}{=}0.
\]
Hence for almost all $\omega\in\Omega$ and all $n\in\b N\cup\left\{ \infty\right\} $
we have for $u\in B_{J_{1},J_{2}}\left(\omega\right)$ that
\[
\xi\left(u\right)\left(\omega\right)_{j}\stackrel{}{=}m_{n}\left(u\right)\left(\omega\right)_{j}=m\left(u\right)_{j}=\begin{cases}
a_{j} & ,j\in J_{1}\setminus J_{2}\\
\t_{j} & ,j\in J_{2}
\end{cases}
\]
for some $a_{j}\in\text{\ensuremath{\mathbb{R}}},$ $a_{j}\neq \t_{j}$.
We will use the notation 
\[
\xi\left(u\right)_{J_{1}^{c}}:=\left(\xi\left(u\right)_{j}\right)_{j\in J_{1}^{c}}
\]
and 
\[
\t_{J_{1}^{c}}:=\left(\t_{j}\right)_{j\in J_{1}^{c}}
\]
for the sub-vectors containing only the elements with index in $J_{1}^{c}$.
Then combining the previous result with the Portmanteau Theorem yields
for almost all $\omega\in\Omega$ for all $u\in B_{J_{1},J_{2}}\left(\omega\right)$
\begin{align*}
P\left(\xi\left(u\right)\geq T|\mathcal{G}_{n}\right)\left(\omega\right)= & P\left(\xi\left(u\right)_{J_{1}^{c}}\geq \t_{J_{1}^{c}}|\mathcal{G}_{n}\right)\left(\omega\right)\prod_{j\in J_{1}\setminus J_{2}}\I_{m\left(u\right)_{j}\geq \t_{j}}\\
\xrightarrow[n\rightarrow\infty]{} & P\left(\xi\left(u\right)_{J_{1}^{c}}\geq \t_{J_{1}^{c}}|\mathcal{G}_{\infty}\right)\left(\omega\right)\prod_{j\in J_{1}\setminus J_{2}}\I_{m\left(u\right)_{j}\geq \t_{j}}\\
= & P\left(\xi\left(u\right)\geq T|\mathcal{G}_{\infty}\right)\left(\omega\right),
\end{align*}
since $\xi\left(u\right)_{J_{1}^{c}}$ is again a Gaussian vector
with almost surely convergent conditional mean and covariance. Indeed,
for all elements in $j\in J_{1}^{c}$ (which also means $j\notin J_{2}$
by assumption) we have $k_{\infty}\left(u,u\right)\left(\omega\right)_{jj}>0$
or $m_{\infty}\left(u\right)\left(\omega\right)_{j}\neq \t_{j}$ which
implies for the boundary of ${\bf T}_{J_{1}^{c}}=\bigtimes_{j\in J_{1}^{c}}\left[\t_{j}\infty\right),$
that
\begin{align*}
P\left(\xi\left(u\right)_{J_{1}^{c}}\in\partial{\bf T}_{J_{1}^{c}}|\mathcal{G}_{\infty}\right)\left(\omega\right) & =P\left(\exists j\in J_{1}^{c}:\xi\left(u\right)_{j}=\t_{j}|\mathcal{G}_{\infty}\right)\left(\omega\right)\\
 & =P\left(\bigcup_{j\in J_{1}^{c}}\left\{ \xi\left(u\right)_{j}=\t_{j}\right\} |\mathcal{G}_{\infty}\right)\left(\omega\right)\\
 & \leq\sum_{j\in J_{1}^{c}}P\left(\xi\left(u\right)_{j}=\t_{j}|\mathcal{G}_{\infty}\right)\left(\omega\right)\\
 & =0
\end{align*}
and hence we can apply the Portmanteau Theorem, which concludes the
second property.
\end{proof}
\begin{proof}
\textbf{(Lemma \ref{lem:6})} \\
It remains to show 
$\left\{ \nu\in\mathbb{M}:\mathcal{H}^{IBV}\left(\nu\right)=0\right\}
\supseteq
\left\{ \nu\in\mathbb{M}:\mathcal{G}^{IBV}\left(\nu\right)=0\right\}$.
Let $\nu=\mathcal{GP}_{d}\left(m,k\right)\in\left\{ \nu\in\mathbb{M}:\mathcal{G}^{IBV}\left(\nu\right)=0\right\}$
and $\xi\sim\nu$, then it holds by the law of total variance (see
Theorem 8.2 in \cite{key-9})
\begin{align*}
0 & =\sup_{x\in\mathbb{X}}\mathcal{G}_{x}^{IBV}\left(\nu\right)\\
 & =\sup_{x\in\mathbb{X}}\mathcal{H}^{IBV}\left(\nu\right)-\mathcal{J}_{x}^{IBV}\left(\nu\right)\\
 & =\sup_{x\in\mathbb{X}}\int_{\mathbb{X}}\text{Var}\left({\bf 1}_{\Gamma\left(\xi\right)}\left(u\right)\right)\mu\left(du\right)-\text{\ensuremath{\b E}}\left[\int_{\mathbb{X}}\text{Var}\left({\bf 1}_{\Gamma\left(\xi\right)}\left(u\right)|Z\left(x\right)\right)\mu\left(du\right)\right]\\
 & =\sup_{x\in\mathbb{X}}\int_{\mathbb{X}}\text{Var}\left({\bf 1}_{\Gamma\left(\xi\right)}\left(u\right)\right)-\b E\left[\text{Var}\left({\bf 1}_{\Gamma\left(\xi\right)}\left(u\right)|Z\left(x\right)\right)\right]\mu\left(du\right)\\
 & =\sup_{x\in\mathbb{X}}\int_{\mathbb{X}}\text{Var}\left(\mathbb{E}\left[{\bf 1}_{\Gamma\left(\xi\right)}\left(u\right)|Z\left(x\right)\right]\right)\mu\left(du\right).
\end{align*}
This means for all $x\in\mathbb{X}$, with $Z(x)=\xi\left(x\right)+\tau\left(x\right)U$,
$U\sim\mathcal{N}_{d}\left(0,I_{d}\right)$ independent of $\xi$,
we have
\begin{align*}
\text{Var}\left(\mathbb{E}\left[{\bf 1}_{\Gamma\left(\xi\right)}\left(u\right)|Z\left(x\right)\right]\right) & =\text{Var}\left(P\left(\xi\left(u\right)\geq T|Z\left(x\right)\right)\right)\\
 & =\text{Var}\left(\mathcal{N}_{d}\left(m_{1}\left(u\right),k_{1}\left(u,u\right)\right)\left({\bf T}\right)\right)\\
 & =0
\end{align*}
for $\mu$-almost all $u\in\X$. Note that $\mathcal{N}_{d}\left(m_{1}\left(u\right),k_{1}\left(u,u\right)\right)$
is a random measure with 
\[
\mathcal{N}_{d}\left(m_{1}\left(u\right),k_{1}\left(u,u\right)\right)\left(\omega\right)=\mathcal{N}_{d}\left(m_{1}\left(u\right)\left(\omega\right),k_{1}\left(u,u\right)\right),
\]
where $m_{1}\left(u\right)$ and $k_{1}\left(u\right)$ are given
by
\[
m_{1}\left(u\right)=m(u)+k(u,x)\Sigma\left(x\right){}^{\dagger}\left(Z\left(x\right)-m\left(x\right)\right),
\]
\[
k_{1}\left(u,u\right)=k\left(u,u\right)-k(u,x)\Sigma\left(x\right){}^{\dagger}k(u,x)^{\top}.
\]
Since $k_{1}\left(u,u\right)$ does not depend on $\omega\in\Omega$
and $\text{Var}\left(\mathcal{N}_{d}\left(m_{1}\left(u\right),k_{1}\left(u,u\right)\right)\left({\bf T}\right)\right)=0$,
we conclude that $m_{1}\left(u\right)$ has to be $P$-almost surely
constant and hence
\begin{align*}
\text{Var}\left(m_{1}\left(u\right)\right) & =\text{Var}\left(k(u,x)\Sigma\left(x\right){}^{\dagger}\left(Z\left(x\right)-m\left(x\right)\right)\right)\\
 & =k(u,x)\Sigma\left(x\right){}^{\dagger}\text{Var}\left(Z\left(x\right)\right)\left(k(u,x)\Sigma\left(x\right){}^{\dagger}\right)^{\top}\\
 & =k(u,x)\Sigma\left(x\right){}^{\dagger}\Sigma\left(x\right)\Sigma\left(x\right){}^{\dagger}k(u,x)^{\top}\\
 & =k(u,x)\Sigma\left(x\right){}^{\dagger}k(u,x)^{\top}\\
 & =0\in\mathbb{R}^{d\times d}.
\end{align*}
Since $\Sigma\left(x\right)$ is symmetric positive semi-definite,
we know that $\Sigma\left(x\right){}^{\dagger}$ is symmetric positive
semi-definite and hence $k(u,x)\Sigma\left(x\right){}^{\dagger}k(u,x)^{\top}=0$
if and only if $k(u,x)\Sigma\left(x\right){}^{\dagger}=0$. Indeed,
if $A\in\Rdd$, $B\in S_{d}^{+}$ and $ABA^{\top}=0$, then
\[
ABA^{\top}=AQQ^{\top}A^{\top}=AQ\left(AQ\right)^{\top}=0
\]
for $Q\in\b R^{d\times d}$ with $B=QQ^{\top}$ and hence $\left\Vert AQ\right\Vert _{F}^{2}=\text{tr}\left(\left(AQ\right)^{\top}AQ\right)=0$
implying $AQ=0$ and finally $AQQ^{\top}=AB=0$ (the other direction
is trivial). Using the properties $\Sigma\left(x\right)\Sigma\left(x\right){}^{\dagger}\Sigma\left(x\right)=\Sigma\left(x\right)$
and $\left(\Sigma\left(x\right){}^{\dagger}\Sigma\left(x\right)\right)^{\top}=\Sigma\left(x\right){}^{\dagger}\Sigma\left(x\right)$
of the pseudo-inverse $\Sigma\left(x\right){}^{\dagger}$, this implies
for every $x\in\b X$
\[
k(u,x)\Sigma\left(x\right)=0\in\mathbb{R}^{d\times d}
\]
for $\mu$-almost all $u\in\b X$. This also yields $k(u,u)\Sigma\left(u\right)=0$
and hence 
\[
\text{tr}\left(k\left(u,u\right)\Sigma\left(u\right)\right)=\text{tr}\left(k\left(u,u\right)^{2}\right)+\text{tr}\left(k\left(u,u\right)\c T\left(u\right)\right)=0
\]
for $\mu$-almost all $u\in\b X$.

We have
\[
\text{tr}\left(k\left(u,u\right)^{2}\right)=\text{tr}\left(k\left(u,u\right)^{\top}k\left(u,u\right)\right)=\left\Vert k\left(u,u\right)\right\Vert _{F}^{2}\geq0
\]
and
\[
\text{tr}\left(k\left(u,u\right)\c T\left(u\right)\right)\geq0,
\]
since $\c T\left(u\right)$ and $k\left(u,u\right)$ are both symmetric
and positive semi-definite. Indeed, we can write for $A,B\in S_{d}^{+}$
that 
\[
\text{tr}\left(AB\right)=\text{tr}\left(AQQ^{\top}\right)=\text{tr}\left(Q^{\top}AQ\right)=\sum_{i=1}^{d}q_{i}^{\top}Aq_{i}\geq0
\]
for $Q\in\b R^{d\times d}$ with $B=QQ^{\top}$ by the Spectral Theorem
(see Theorem 1.3.1 in \cite{key-18}). 

We conclude
\[
\left\Vert k\left(u,u\right)\right\Vert _{F}^{2}=0
\]
and hence $k\left(u,u\right)=0$ for $\mu$-almost all $u\in\b X$,
which yields $P\left(\xi\left(u\right)\geq T\right)={\bf 1}_{m\left(u\right)\geq T}$
and finally
\begin{align*}
\mathcal{H}^{IBV}\left(\nu\right) & =\int_{\mathbb{X}}\text{Var}\left({\bf 1}_{\Gamma\left(\xi\right)}\left(u\right)\right)\mu\left(du\right)\\
 & =\int_{\mathbb{X}}P\left(\xi\left(u\right)\geq T\right)\left(1-P\left(\xi\left(u\right)\geq T\right)\right)\mu\left(du\right)\\
 & =0.
\end{align*}
\end{proof}
\begin{proof}
\textbf{(Theorem \ref{thm:9})} \\
The first statement follows by combining the three Lemmas in Section \ref{subsec:IBV} with Proposition \ref{prop:7}. For the second part note that
as seen in the proof of Lemma \ref{lem:13} we have $k_{\infty}\left(u,u\right)\left(\omega\right)=0\in\Rdd$
for $\mu$-almost all $u\in\mathbb{X}$, since $\mathcal{H}^{IBV}\left(P_{\infty}^{\xi}\left(\text{\ensuremath{\omega}}\right)\right)=0$
for almost all $\omega\in\Omega$ by the previous Theorem. Furthermore,
we have for the random set 
\[
A=\bigcup_{\underset{J_2\subseteq J_1}{J_1,J_2\subseteq\left\{ 1,...,d\right\} }}B_{J_1,J_2}
\]
that 
\[
\mu\left(\mathbb{X}\right)\overset{}{=}\mu\left(A\left(\omega\right)\right)
\]
for almost all $\omega\in\Omega$ and
\[
p_{n}\left(u\right)\left(\omega\right)=P\left(\xi\left(u\right)\geq T|\mathcal{F}_{n}\right)\left(\omega\right)\xrightarrow[n\rightarrow\infty]{}P\left(\xi\left(u\right)\geq T|\mathcal{F}_{\infty}\right)\left(\omega\right)=p_{\infty}\left(u\right)\left(\omega\right)
\]
for all $u\in A\left(\omega\right)$ by Lemma \ref{lem:13}.
Combining both statements we conclude
\[
p_{n}\left(u\right)\left(\omega\right)\xrightarrow[n\rightarrow\infty]{}{\bf 1}_{\xi\left(u\right)\left(\omega\right)\geq T}
\]
for $P$-almost all $\omega\in\Omega$ and $\mu$-almost all $u\in\b X$,
since $k_{\infty}\left(u,u\right)\overset{a.s.}{=}0$ yields $p_{\infty}\left(u\right)=\mathbb{E}\left[{\bf 1}_{\Gamma\left(\xi\right)}\left(u\right)|\mathcal{F}_{\infty}\right]\overset{a.s.}{=}{\bf 1}_{\Gamma\left(\xi\right)}\left(u\right)$.
The claim follows by using the Dominated Convergence Theorem.
\end{proof}

\subsubsection{Proofs of Section 5.2 for EMV}

\begin{proof}
\textbf{(Lemma \ref{lem:7})} \\
$\mathcal{H}^{EMV}$ is $\mathfrak{P}$-uniformly integrable since
we have for every Gaussian random element $\xi$ in $\b S$ 
\begin{align*}
\text{Var}\left(\alpha\left(\xi\right)\right) & =\b E\left[\alpha\left(\xi\right)^{2}\right]-\b E\left[\alpha\left(\xi\right)\right]^{2}\\
 & \leq\b E\left[\alpha\left(\xi\right)^{2}\right]\\
 & \leq\mu\left(\b X\right)^{2},
\end{align*}
since $\alpha\left(\xi\right)=\mu\left(\Gamma\left(\xi\right)\right)$
and $\Gamma\left(\xi\right)\subseteq\b X$ almost surely. Hence we
have for any measure $\nu\in\mathbb{M}$ taking $\xi\sim\nu$ that
\[
\left|\mathcal{H}^{EMV}\left(\nu\right)\right|=\text{Var}\left(\alpha\left(\xi\right)\right)\leq\mu\left(\b X\right)^{2},
\]
since $\mu$ is a finite measure over $\mathbb{X}$. 

$H_{n}^{EMV}$ is $\c F_{n}$-measurable and integrable by definition
of the conditional expectation. Furthermore, it holds
\begin{align*}
\b E\left[H_{n}^{EMV}|\c F_{n-1}\right] & \overset{a.s.}{\leq}H_{n-1}^{EMV},
\end{align*}
since the conditional variance is a supermartingale by Jensen's inequality
and the tower property. Indeed, 
\begin{align*}
\b E\left[\text{Var}\left(\alpha\left(\xi\right)|\c F_{n}\right)|\c F_{n-1}\right] & =\b E\left[\b E\left[\alpha\left(\xi\right)^{2}|\c F_{n}\right]-\b E\left[\alpha\left(\xi\right)|\c F_{n}\right]^{2}|\c F_{n-1}\right]\\
 & =\b E\left[\alpha\left(\xi\right)^{2}|\c F_{n-1}\right]-\b E\left[\b E\left[\alpha\left(\xi\right)|\c F_{n}\right]^{2}|\c F_{n-1}\right]\\
 & \overset{a.s.}{\leq}\b E\left[\alpha\left(\xi\right)^{2}|\c F_{n-1}\right]-\b E\left[\alpha\left(\xi\right)|\c F_{n-1}\right]^{2}\\
 & =\text{Var}\left(\alpha\left(\xi\right)|\c F_{n-1}\right).
\end{align*}
\end{proof}
\begin{proof}
\textbf{(Lemma \ref{lem:14})} \\
Let $\xi$ be a random element in $\b S$ with $\xi\sim\c{GP}_{d}\left(m,k\right)$
and let $\bm{\nu}_{n}\in\mathfrak{P}\left(\xi\right)$ for all $n\in\mathbb{N}\cup\left\{ \infty\right\} $
such that $\bm{\nu}_{n}\xrightarrow{}\bm{\nu}_{\infty}$
almost surely. By definition of $\mathfrak{P}\left(\xi\right)$ there
exist $\sigma$-algebras $\c G_{n}$ such that $\bm{\nu}_{n}=P\left(\xi\in\cdot|\mathcal{G}_{n}\right)$
for all $n\in\mathbb{N}\cup\left\{ \infty\right\} $. 

Note first that for almost all $\omega\in\Omega$ and all $n\in\b N\cup\left\{ \infty\right\} $
we have for $u_{i}\in B_{J_{1}^{i},J_{2}^{i}}\left(\omega\right)$,
$j\in J_{1}^{i}$ and $i\in\left\{ 1,2\right\} $
\[
\xi\left(u_{i}\right)\left(\omega\right)_{j}\stackrel{}{=}m_{n}\left(u_{i}\right)\left(\omega\right)_{j}=m\left(u_{i}\right)_{j}=\begin{cases}
a_{j}^{i} & ,j\in J_{1}^{i}\setminus J_{2}^{i}\\
\t_{j} & ,j\in J_{2}^{i}
\end{cases}
\]
for some $a_{j}^{i}\in\text{\ensuremath{\mathbb{R}} }$ with $a_{j}^{i}\neq \t_{j}$.

Combining this with the Portmanteau Theorem leads to
\begin{align*}
\b E\left[{\bf 1}_{\Gamma\left(\xi\right)}\left(u_{i}\right)|\c G_{n}\right]\left(\omega\right)= & P\left(\xi\left(u_{i}\right)\geq T|\mathcal{G}_{n}\right)\left(\omega\right)\\
= & P\left(\xi\left(u_{i}\right)_{\left(J_{1}^{i}\right)^{c}}\geq \t_{\left(J_{1}^{i}\right)^{c}}|\mathcal{G}_{n}\right)\left(\omega\right)\prod_{j\in J_{1}^{i}\setminus J_{2}^{i}}\I_{m\left(u_{i}\right)_{j}\geq \t_{j}}\\
\xrightarrow[n\rightarrow\infty]{} & P\left(\xi\left(u\right)_{\left(J_{1}^{i}\right)^{c}}\geq \t_{\left(J_{1}^{i}\right)^{c}}|\mathcal{G}_{\infty}\right)\left(\omega\right)\prod_{j\in J_{1}^{i}\setminus J_{2}^{i}}\I_{m\left(u_{i}\right)_{j}\geq \t_{j}}\\
= & P\left(\xi\left(u_{i}\right)\geq T|\mathcal{G}_{\infty}\right)\left(\omega\right)\\
= & \b E\left[{\bf 1}_{\Gamma\left(\xi\right)}\left(u_{i}\right)|\c G_{\infty}\right]\left(\omega\right)
\end{align*}
for almost all $\omega\in\Omega$ for all $u_{i}\in B_{J_{1}^{i},J_{2}^{i}}\left(\omega\right)$
and $i\in\left\{ 1,2\right\} $, as we have already seen in the second
claim in the proof of Lemma \ref{lem:5}. We conclude 
\begin{align*}
    & \b E\left[{\bf 1}_{\Gamma\left(\xi\right)}\left(u_{1}\right)|\c G_{n}\right]\left(\omega\right)\b E\left[{\bf 1}_{\Gamma\left(\xi\right)}\left(u_{2}\right)|\c G_{n}\right]\left(\omega\right) \\ \xrightarrow[n\rightarrow\infty]{} & \b E\left[{\bf 1}_{\Gamma\left(\xi\right)}\left(u_{1}\right)|\c G_{\infty}\right]\left(\omega\right)\b E\left[{\bf 1}_{\Gamma\left(\xi\right)}\left(u_{2}\right)|\c G_{\infty}\right]\left(\omega\right).
\end{align*}

Similarly, we have again by the Portmanteau Theorem 
\begin{align*}
 & \b E\left[{\bf 1}_{\Gamma\left(\xi\right)}\left(u_{1}\right){\bf 1}_{\Gamma\left(\xi\right)}\left(u_{2}\right)|\c G_{n}\right]\left(\omega\right)\\
= & P\left(\xi\left(u_{1}\right)\geq T,\xi\left(u_{2}\right)\geq T|\c G_{n}\right)\left(\omega\right)\\
= & P\left(\xi\left(u_{1}\right)_{\left(J_{0}^{1}\right)^{c}}\geq \t_{\left(J_{1}^{1}\right)^{c}},\xi\left(u_{2}\right)_{\left(J_{1}^{2}\right)^{c}}\geq \t_{\left(J_{1}^{2}\right)^{c}}|\c G_{n}\right)\left(\omega\right)\prod_{\underset{i=1,2}{j\in J_{1}^{i}\setminus J_{2}^{i}}}\I_{m\left(u_{i}\right)_{j}\geq \t_{j}}\\
\xrightarrow[n\rightarrow\infty]{} & P\left(\xi\left(u_{1}\right)_{\left(J_{1}^{1}\right)^{c}}\geq \t_{\left(J_{1}^{1}\right)^{c}},\xi\left(u_{2}\right)_{\left(J_{1}^{2}\right)^{c}}\geq \t_{\left(J_{1}^{2}\right)^{c}}|\c G_{\infty}\right)\left(\omega\right)\prod_{\underset{i=1,2}{j\in J_{1}^{i}\setminus J_{2}^{i}}}\I_{m\left(u_{i}\right)_{j}\geq \t_{j}}\\
= & P\left(\xi\left(u_{1}\right)\geq T,\xi\left(u_{2}\right)\geq T|\c G_{\infty}\right)\left(\omega\right)\\
= & \b E\left[{\bf 1}_{\Gamma\left(\xi\right)}\left(u_{1}\right){\bf 1}_{\Gamma\left(\xi\right)}\left(u_{2}\right)|\c G_{\infty}\right]\left(\omega\right),
\end{align*}
since $\left(\text{\ensuremath{\xi\left(u_{1}\right)^{\top},\xi\left(u_{2}\right)^{\top}}}\right)^{\top}$
is a Gaussian vector and
\begin{align*}
& P\left(\bigcup_{i=1,2}\left\{ \xi\left(u_{i}\right)_{\left(J_{1}^{i}\right)^{c}}\in\partial{\bf T}_{\left(J_{1}^{i}\right)^{c}}\right\} |\mathcal{G}_{\infty}\right)\left(\omega\right) \\ \leq & \sum_{i=1,2}P\left(\xi\left(u_{i}\right)_{\left(J_{1}^{i}\right)^{c}}\in\partial{\bf T}_{\left(J_{1}^{i}\right)^{c}}|\mathcal{G}_{\infty}\right)\left(\omega\right)\\
= & 0.
\end{align*}
Combining both statements yields for almost all $\omega\in\Omega$
and all $u_{i}\in B_{J_{1}^{i},J_{2}^{i}}\left(\omega\right)$,
$i=1,2$
\[
\text{Cov}\left({\bf 1}_{\Gamma\left(\xi\right)}\left(u_{1}\right),{\bf 1}_{\Gamma\left(\xi\right)}\left(u_{2}\right)|\c G_{n}\right)\left(\omega\right)\xrightarrow[n\rightarrow\infty]{}\text{Cov}\left({\bf 1}_{\Gamma\left(\xi\right)}\left(u_{1}\right),{\bf 1}_{\Gamma\left(\xi\right)}\left(u_{2}\right)|\c G_{\infty}\right)\left(\omega\right).
\]
\end{proof}
\begin{proof}
\textbf{(Lemma \ref{lem:aux})} \\
Assume without loss of generality that all random elements $U,V$
and $W$ are centered. By Fernique's Theorem
\[
\b E\left[\text{exp}\left(c\left\Vert V\right\Vert \right)\right]=\int_{\b R^{d}}\text{exp}\left(c\left\Vert x\right\Vert \right)P^{V}\left(dx\right)<\infty
\]
 for some $c>0$ and hence by Theorem 3.2.18 in \cite{key-36} the
space of polynomials $\Pi^{d}$ is dense in the space $L^{2}\left(\b R^{d},\c B\left(\b R^{d}\right),dP^{V}\right)$.
The polynomials in $\Pi^{d}$ are here given by 
\[
p\left(x\right)=\sum_{\alpha\in A}c_{\alpha}x^{\alpha},
\]
where $A$ is some (finite) subset of $\b N_{0}^{d}$ and the product
$x^{\alpha}=x_{1}^{\alpha_{1}}...x_{d}^{\alpha_{d}}$ is called monomial
in the variables $x_{1},...,x_{d}$ for $\alpha\in\b N_{0}^{d}$ and
$c_{\alpha}\in\b R$. The integer $\left|\alpha\right|=\alpha_{1}+...+\alpha_{d}$
is called total degree of the monomial $x^{\alpha}$. This means every
$f\in L^{2}\left(\b R^{d},\c B\left(\b R^{d}\right),dP^{V}\right)$
can be approximated by a sequence of polynomials $\left(x\mapsto\sum_{\alpha\in A_{n}}c_{\alpha}x^{\alpha}\right)_{n\in\b N}$
in $\Pi^{d}$, i.e. 
\[
\int_{\mathbb{R}^{d}}\left(f\left(x\right)-\sum_{\alpha\in A_{n}}c_{\alpha}x^{\alpha}\right)^{2}P^{V}\left(dx\right)\xrightarrow[n\rightarrow\infty]{}0.
\]
By the factorization Lemma every element in $Z\in L^{2}\left(\Omega,\sigma\left(V\right),P\right)$
can be written as $f\left(V\right)$ for some measurable function
$f:\b R^{d}\rightarrow\b R$ with 
\[
\b E\left[f\left(V\right)^{2}\right]=\int_{\mathbb{R}^{d}}f\left(x\right)^{2}P^{V}\left(dx\right)<\infty.
\]
Hence the above statement yields the existence of a polynomial sequence
such that 
\[
\b E\left[\left(Z-\sum_{\alpha\in A_{n}}c_{\alpha}V^{\alpha}\right)^{2}\right]=\b E\left[\left(f\left(V\right)-\sum_{\alpha\in A_{n}}c_{\alpha}V^{\alpha}\right)^{2}\right]\xrightarrow[n\rightarrow\infty]{}0.
\]

Assume now that $U$ is not orthogonal to $L^{2}\left(\Omega,\sigma\left(V\right),P\right)$.
Then there exists an element $Z\in L^{2}\left(\Omega,\sigma\left(V\right),P\right)$
such that $\b E\left[UZ\right]\neq0$ and since 
\[
\left|\b E\left[UZ\right]-\b E\left[U\sum_{\alpha\in A_{n}}c_{\alpha}V^{\alpha}\right]\right|\leq\b E\left[U^{2}\right]^{\frac{1}{2}}\underbrace{\b E\left[\left(Z-\sum_{\alpha\in A_{n}}c_{\alpha}V^{\alpha}\right)^{2}\right]^{\frac{1}{2}}}_{\xrightarrow[n\rightarrow\infty]{}0}
\]
by Cauchy-Schwarz there must exist $N\in\mathbb{N}$ such that for
all $n\geq N$ we have $\b E\left[U\sum_{\alpha\in A_{n}}c_{\alpha}V^{\alpha}\right]\neq0$
and by the linearity of the expectation there must hence be a multi-index
$\alpha$ with total degree $\left|\alpha\right|=k_{0}$ for some
$k_{0}\in\b N_{0}$ such that $\b E\left[UV^{\alpha}\right]\neq0$.

This means the set 
\[
I_{k_{0}}:=\left\{ \alpha\in\mathbb{N}_{0}^{d}:\b E\left[UV^{\alpha}\right]\neq0,\left|\alpha\right|=k_{0}\right\} 
\]
is non-empty. By reducing $k_{0}$ by one in each step, we can calculate
the sets $I_{k}$ for every $0\leq k\leq k_{0}$ and define by $k_{*}$
the smallest $k$ such that $I_{k}$ is non-empty. Then we have for
every multi-index $\lambda\in I_{k_{*}}$ 
\[
\mathbb{E}\left[UV^{\lambda}\right]=\mathbb{E}\left[U\prod_{i=1}^{d}V_{i}^{\lambda_{i}}\right]\neq0
\]
 and
\[
\mathbb{E}\left[U\prod_{i=1}^{d}V_{i}^{k_{i}}\right]=0,
\]
if $k_{i}\leq\lambda_{i}$ for all $i\in\left\{ 1,...,n\right\} $
with at least one strict inequality. Indeed, if this would not be
the case we would have a contradiction with the minimality of $k_{*}$. 

Using the independence, we conclude 
\begin{align*}
\mathbb{E}\left[U\prod_{i=1}^{d}\left(V+W\right)_{i}^{\lambda_{i}}\right] & =\mathbb{E}\left[U\prod_{i=1}^{d}\left(\sum_{k=0}^{\lambda_{i}}\binom{\lambda_{i}}{k}V_{i}^{k}W_{i}^{\lambda_{i}-k}\right)\right]\\
 & =\mathbb{E}\left[U\sum_{k_{1},...,k_{d}=0}^{\lambda_{1},...,\lambda_{d}}\prod_{i=1}^{d}\binom{\lambda_{i}}{k_{i}}V_{i}^{k_{i}}W_{i}^{\lambda_{i}-k_{i}}\right]\\
 & =\sum_{k_{1},...,k_{d}=0}^{\lambda_{1},...,\lambda_{d}}\prod_{i=1}^{d}\binom{\lambda_{i}}{k_{i}}\mathbb{E}\left[U\prod_{i=1}^{d}V_{i}^{k_{i}}\right]\b E\left[\prod_{i=1}^{d}W_{i}^{\lambda_{i}-k_{i}}\right]\\
 & =\mathbb{E}\left[U\prod_{i=1}^{d}V_{i}^{\lambda_{i}}\right]\\
 & \neq0
\end{align*}
and hence $U$ can not be orthogonal to $L^{2}\left(\Omega,\sigma\left(V+W\right),P\right)$.
\end{proof}



\bibliographystyle{imsart-number} 
\bibliography{multiSUR_references}       


\end{document}